\documentclass[aos]{imsart}

\usepackage{graphicx} 

\usepackage{fancyvrb}
\usepackage[dvipsnames]{xcolor}
\usepackage{lmodern}

\usepackage{environ}

\definecolor{MidnightBlue}{rgb}{0.1, 0.1, 0.44}

\usepackage{hyperref}

\hypersetup{
    colorlinks=true,       
    linkcolor=MidnightBlue, 
    citecolor=MidnightBlue, 
    filecolor=MidnightBlue, 
    urlcolor=MidnightBlue   
}

\RequirePackage{amsmath, amsthm, amssymb, mathtools}
\RequirePackage{comment}
\RequirePackage{graphicx}
\usepackage{caption, subcaption}

\RequirePackage[scale=0.95]{sourcecodepro}
\RequirePackage{mathpazo}

\RequirePackage{dsfont}
\RequirePackage{parskip}
\RequirePackage{tikz}
    \usetikzlibrary{arrows.meta}
    \usetikzlibrary{positioning}
    \usetikzlibrary{decorations.markings, calc, arrows} 
    \usetikzlibrary{plotmarks, patterns}
\RequirePackage{microtype}
\usepackage{pgfplots}
\pgfplotsset{compat=1.17}

\usepackage{hyperref}

\usepackage{stfloats}

\usepackage{newfloat}
\DeclareFloatingEnvironment[
    fileext=los,
    listname={List of Algorithms},
    name=Listing,
    placement=tbhp,
    within={},
]{mylisting}

\usepackage{algorithm}
\theoremstyle{plain}
\newtheorem{prototheorem}{theorem}
\newtheorem{theorem}[prototheorem]{Theorem}

\theoremstyle{plain}

\theoremstyle{plain}
\newtheorem{prototheorem3}{theorem}
\newtheorem{lemma}[prototheorem3]{Lemma}

\theoremstyle{remark}

\newcommand{\R}{\mathbb{R}}

\newcommand{\W}{\boldsymbol{\mathcal{W}}}

\newcommand{\Unif}{\mathrm{Unif}}

\newcommand{\NURS}{\mathrm{NURS}}

\newcommand{\slice}{\mathrm{slice}}
\newcommand{\orbit}{\mathrm{Orbit}}
\newcommand{\ext}{\mathrm{ext}}
\newcommand{\old}{\mathrm{old}}

\newcommand{\gHR}{\mathrm{gH\&R}}
\newcommand{\C}{\mathcal{C}}
\newcommand{\RWM}{\mathrm{RWM}}

\newcommand{\lb}[1]{{\lfloor #1 \rfloor_h}}

\newcommand{\calO}{\mathcal{O}}

\newcommand{\rd}{\mathrm{d}}

\newcommand{\leftmost}{\mathrm{left}}
\newcommand{\rightmost}{\mathrm{right}}

\newcommand{\indc}[1]{{\mathbf{1}_{\left\{{#1}\right\}}}}
\newcommand{\Indc}[1]{{\mathbf{1}\left\{{#1}\right\}}}
\newcommand{\cat}{{\mathrm{categorical}}}

\begin{document}

\begin{frontmatter}
\title{The No-Underrun Sampler: A Locally Adaptive, Gradient-Free MCMC Method}
\runtitle{Locally Adaptive, Gradient-Free MCMC}

\begin{aug}
\author[A]{\fnms{Nawaf}~\snm{Bou-Rabee}\ead[label=e1]{nawaf.bourabee@rutgers.edu}},
\author[B]{\fnms{Bob}~\snm{Carpenter}\ead[label=e2]{bcarpenter@flatironinstitute.org}},
\author[C]{\fnms{Sifan}~\snm{Liu}\ead[label=e3]{sliu@flatironinstitute.org}},
\and
\author[D]{\fnms{Stefan}~\snm{Oberd\"orster}\ead[label=e4]{oberdoerster@uni-bonn.de}}

\address[A]{Department of Mathematical Sciences, Rutgers University\printead[presep={ ,\ }]{e1}}
\address[B]{Center for Computational Mathematics, Flatiron Institute\printead[presep={ ,\ }]{e2}}
\address[C]{Center for Computational Mathematics, Flatiron Institute\printead[presep={ ,\ }]{e3}}
\address[D]{Institute for Applied Mathematics, University of Bonn\printead[presep={,\ }]{e4}}
\end{aug}

\begin{abstract}
In this work, we introduce the No-Underrun Sampler (NURS), a locally-adaptive, gradient-free Markov chain Monte Carlo method that blends ideas from Hit-and-Run and the No-U-Turn Sampler.   NURS dynamically adapts to the local scale of the target distribution without requiring gradient evaluations, making it especially suitable for applications where gradients are unavailable or costly. We establish key theoretical properties, including reversibility, formal connections to Hit-and-Run and Random Walk Metropolis, Wasserstein contraction comparable to Hit-and-Run in Gaussian targets, and bounds on the total variation distance between the transition kernels of Hit-and-Run and NURS. Empirical experiments, supported by theoretical insights, illustrate the ability of NURS to sample from Neal's funnel, a challenging multi-scale distribution from Bayesian hierarchical inference.
\end{abstract}


\end{frontmatter}

\maketitle

\section{Introduction}

The No-U-Turn Sampler (NUTS) has become one of the most widely used Markov Chain Monte Carlo (MCMC) methods for sampling from continuous high-dimensional distributions \cite{HoGe2014,carpenter2016stan,salvatier2016probabilistic,nimble-article:2017,ge2018t,phan2019composable}.  It leverages a leapfrog approximation of Hamiltonian flow, which relies on the gradient of the log-density to efficiently navigate the geometry of the target distribution \cite{betancourt2017conceptual}. While highly effective, this reliance on gradients limits its applicability in scenarios where derivatives are difficult or impossible to compute. 

Such scenarios are common in fields like natural sciences and engineering, where likelihoods often depend on computationally expensive solvers for partial differential equations or are implemented in large legacy codebases (e.g., in Fortran or C) that are incompatible with automatic differentiation. As a result, gradient-free methods remain the standard in many practical applications, such as astrophysics, where ensemble methods dominate \cite{foreman2013emcee,goodman2010ensemble}.

This raises an important question: Can the key ideas behind NUTS be adapted to create a gradient-free method? We show that they can with the No-Underrun Sampler (NURS), a new gradient-free MCMC method that blends NUTS’s orbit-based exploration with the simplicity of gradient-free methods. NURS builds its orbits using constant velocity motion, making it both gradient-free and parallelizable. In short, NURS adjusts dynamically to the local geometry of the target distribution—just like NUTS—but without needing gradients.

NURS is closely related to the Hit-and-Run method, a classic sampling algorithm that explores a target distribution by sampling along randomly chosen lines \cite{Diaconis2007HitandRun,chewi2022query,ascolani2024,BoEbOb2024}. While Hit-and-Run has strong theoretical properties \cite{BoEbOb2024}, its exact implementation is often computationally impractical because it requires exact sampling from the target distribution restricted to these lines. NURS addresses this challenge by approximating the sampling process along lines, making the method more practical while preserving key advantages of Hit-and-Run.

The remainder of this section provides an overview of ideal and practical slice sampling along with sliced and slice-free (or multinomial) NUTS.  NURS evolved from this chain of development as a gradient-free counterpart to NUTS and an implementable form of the Hit-and-Run sampler.

\paragraph*{Slice sampling}

The slice sampler reduces the problem of sampling from an absolutely continuous target distribution $\mu$ on $\mathbb R^d$   to alternately sampling from two  uniform distributions.  Denote the target's density also by $\mu$ and for $y \ge 0$ define
\[ \slice(y)\ =\ \bigr\{\theta\in\mathbb R^d\ :\ \mu(\theta) > y \bigr\}\;. \]
The slice sampler generates a Markov chain with transition step given in Algorithm~\ref{alg:slice}.

\begin{algorithm}[ht]
\caption{Slice Sampler: $\ \theta'\sim\pi_{\mathrm{slice}}(\theta,\cdot)$}\label{algo:slice}
\begin{description}
\item[Inputs:] $\theta \in \mathbb{R}^d$ (current state);  \ \ $\mu$ (non-normalized target density on $\mathbb{R}^d$)
\vspace*{2pt}
\hrule
\end{description}
\begin{enumerate}
\item Sample height variable $y\sim\mathrm{Unif}\bigr([0,\,\mu(\theta)]\bigr)$.
\item Sample new state $\theta'\sim\Unif\bigr(\slice(y)\bigr)$.
\end{enumerate}
\label{alg:slice}
\end{algorithm}

This sampler, illustrated in Figure~\ref{fig:slice_sampler}, is equivalent to a systematic-scan Gibbs sampler targeting the uniform distribution over the positive sub-graph of $\mu$ defined by
\[ \mathcal{G}^+(\mu)\ =\ \bigr\{ (\theta,y) \in \mathbb{R}^d\times\mathbb R\ :\ 0 \le y \le \mu(\theta) \bigr\} \;. \] 
Under this interpretation, the slice sampler alternates between sampling $y \sim \Unif\bigr([0,\mu(\theta)]\bigr)$ and $\theta \sim \Unif\bigr(\slice(y)\bigr)$, corresponding to the conditional updates in a Gibbs sampler \cite{mira2003slice}. The correctness of slice sampling is established by showing that if $(\theta,y) \sim \Unif\bigr(\mathcal{G}^+(\mu)\bigr)$, then the marginal distribution of $\theta$ has density proportional to $\mu$ \cite{chen1998toward,RobertCassella}.   With uniform conditionals, it is also straightforward to show that the transition kernel of the slice sampler is reversible with respect to $\mu$.

\begin{figure}
\centering
\includegraphics[width=0.95\textwidth]{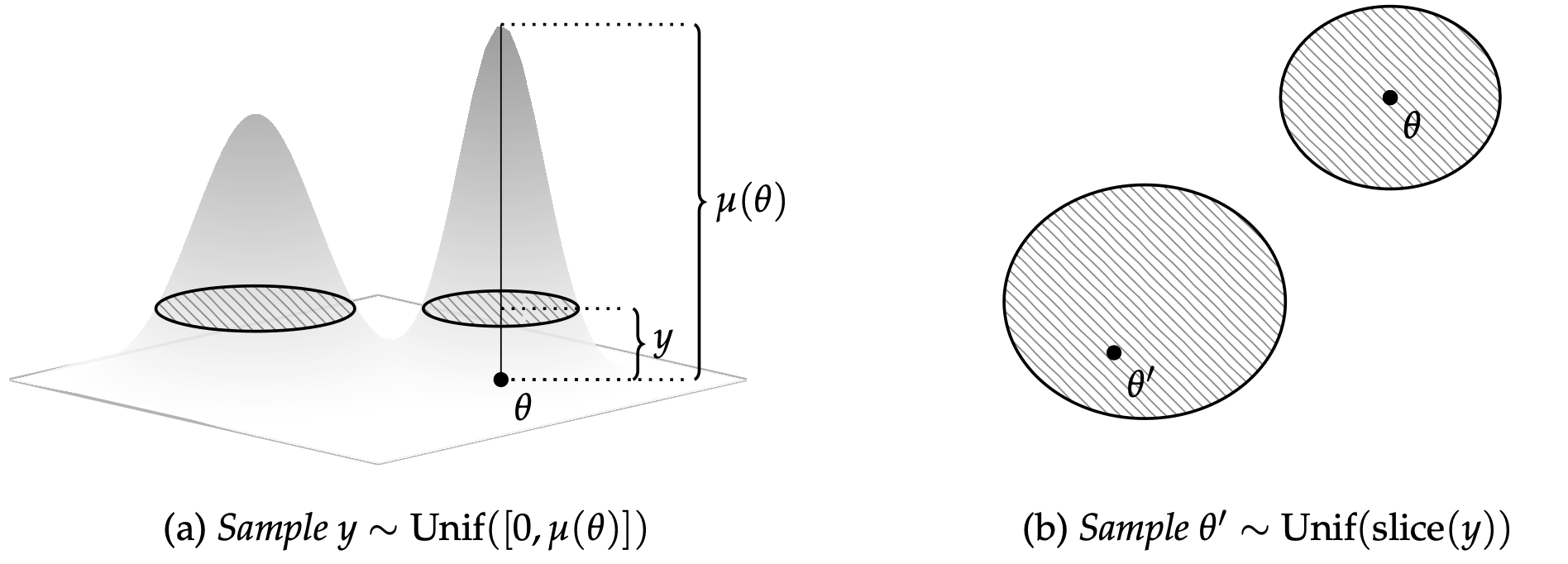}
\caption{\textit{Starting from state $\theta$, the slice sampler first samples a height $y$ uniformly from $[0,\mu(\theta)]$, as shown in (a). It then samples a new state uniformly from the slice of state space where the (non-normalized) density $\mu$ exceeds this height, the gray-shaded region in (b).}}
\label{fig:slice_sampler}
\end{figure}

The slice sampler is appealing because it has no tuning parameters and enables global moves.   However,  directly sampling from the uniform distribution over the slice is often computationally infeasible in practice. The slice sampler is formalized in \cite{neal1997markov,Neal2003Slice}, and under mild regularity conditions, it is provably geometrically ergodic \cite{roberts1999convergence,mira2002efficiency}.   Recent studies have further analyzed its convergence properties, including Wasserstein contraction rates and spectral gaps  \cite{natarovskii2021quantitative,schar2023wasserstein}.

\paragraph*{Neal's hybrid slice sampler}

Since direct sampling from the slice is typically intractable, Neal introduced \emph{hybrid slice samplers} to approximate Step 2 of Algorithm \ref{algo:slice}.  These methods replace exact sampling from the slice with Markov chains that leave the uniform distribution over the slice invariant \cite{Neal2003Slice,murray2010elliptical,Łatuszyński_Rudolf_2024}.  Specifically, subsets of state space are constructed locally, restricted to which sampling from the slice is feasible.  Neal proposed several strategies for selecting such subsets, including the stepping-out and doubling procedures. We briefly review Neal's doubling procedure in one dimension (see Figure~\ref{fig:slice_doubling}), since it is a precursor to the doubling procedure in NUTS.

Given an initial state $\theta \in \mathbb{R}$, a width parameter $w>0$, and a height variable $y$, the doubling procedure builds a manageable subset as follows:
\begin{enumerate}
    \item Start with an interval of size $w$ that is uniformly randomly positioned around the current state $\theta$.  
    \item Expand the interval iteratively by doubling its length at each iteration, with equal probability of expanding to the left or right, until both endpoints lie outside the slice (see Figure~\ref{fig:slice_doubling}).
\end{enumerate}
Given the final interval produced by the doubling procedure, the next state $\theta'$ is selected uniformly from the intersection of this interval and $\slice(y)$.  This procedure is reversible with respect to the uniform distribution on the slice defined by the height variable $y$.

Recent analyses based on Dirichlet forms highlight the trade-offs between practicality and theoretical efficiency of hybrid slice samplers \cite{power2024weak}.   Extensions such as polar and elliptical slice samplers address challenges in specific high-dimensional settings \cite{roberts2002polar,murray2010elliptical}.  Understanding the convergence properties of these methods remains an active area of research \cite{pmlr-v202-schar23a,schar2023wasserstein,rudolf2024dimension,Łatuszyński_Rudolf_2024}.

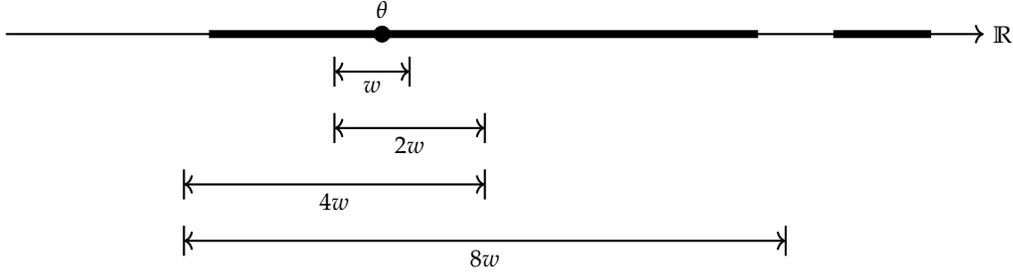
\begin{figure}\centering
\begin{tikzpicture}[scale=1]
\pgfmathsetseed{123}
\draw[thick,->] (-5,0) -- (8,0) node[right] {$\mathbb{R}$};
\draw[line width=3pt] (-2.3,0.) -- (5.,0.);  
\draw[line width=3pt] (6,0.) -- (7.3,0.);  
\filldraw[black] (0,0) circle (3pt) node[above, yshift=0.1cm] {$\theta$};
\pgfmathsetmacro{\U}{rnd}
\pgfmathsetmacro{\x}{0}
\pgfmathsetmacro{\vp}{-0.3}
\pgfmathsetmacro{\vmid}{-0.5}
\pgfmathsetmacro{\vm}{-0.7}
\pgfmathsetmacro{\a}{-\U}
\pgfmathsetmacro{\b}{\a+1}
\draw[thick, <->] (\a,\vmid) -- (\b,\vmid) node[below, midway] {$w$};
\draw[thick] (\a,\vp) -- (\a,\vm);
\draw[thick] (\b,\vp) -- (\b,\vm);
\foreach \i in {1,...,3} {
    \pgfmathsetmacro{\vp}{\vp-0.75};
    \xdef\vp{\vp}
    \pgfmathsetmacro{\vm}{\vm-0.75};
    \xdef\vm{\vm}
    \pgfmathsetmacro{\vmid}{\vmid-0.75};
    \xdef\vmid{\vmid}
    \pgfmathsetmacro{\U}{rnd}
    \ifdim  \U pt > 0.5 pt 
        \pgfmathsetmacro{\a}{\a-(\b-\a)}
        \xdef\a{\a}
    \else 
        \pgfmathsetmacro{\b}{\b+(\b-\a)}
        \xdef\b{\b}
    \fi
     \pgfmathtruncatemacro{\d}{2^(\i)}
    \draw[thick, <->] (\a,\vmid) -- (\b,\vmid) node[below, midway] {$\d w$};
    \draw[thick] (\a,\vp) -- (\a,\vm);
    \draw[thick] (\b,\vp) -- (\b,\vm);
}
\end{tikzpicture}
\caption{\textit{Starting from a state $\theta$, Neal's doubling procedure begins by uniformly at random selecting an interval of size $w$ that contains $\theta$. The interval is then recursively doubled, expanding either to the left or right with equal probability. This doubling continues until both endpoints lie outside the slice, represented by the two bold line segments.}}
\label{fig:slice_doubling}
\end{figure}

\paragraph*{Evolution of slice-based NUTS from Neal's hybrid slice sampler}
The No-U-Turn Sampler (NUTS) \cite{HoGe2014,betancourt2017conceptual, carpenter2016stan,NUTS1}, introduced by Hoffman and Gelman in 2011, uses ideas from Neal's hybrid slice sampler to locally adapt the path length in Hamiltonian Monte Carlo \cite{DuKePeRo1987,Ne2011}.  In particular, NUTS can be interpreted as a hybrid slice sampler in phase space $\mathbb R^{2d}$ where the position variable $\theta \in \mathbb{R}^d$ is augmented by a velocity variable $\rho\in \mathbb{R}^d$.  

Slices are then taken from a product of the target measure $\mu$ and a Gaussian meaure in $\mathbb R^d$, e.g., the product measure with non-normalized density $e^{-H(\theta,\,\rho)}$ defined by the Hamiltonian function $H: \mathbb{R}^{2d} \to \mathbb{R}$,
\[ H(\theta, \rho)\ =\ -\log \mu(\theta)\ +\ \frac{1}{2} |\rho|^2\;. \]
Given the current state $\theta \in \mathbb{R}^d$, each NUTS transition step samples a velocity $\rho \sim \mathcal{N}(0,I_d)$ and a height variable $y \sim \Unif\bigr([0, e^{-H(\theta,\,\rho)}]\bigr)$.  The ingenuity of NUTS lies in how it constructs a manageable subset of $\slice(y) \subset \mathbb{R}^{2d}$ to sample from.  NUTS introduces the concept of an ``orbit'' built by the leapfrog discretization of the Hamiltonian flow associated with $H$.  An essential property of this leapfrog integrator is that it is volume-preserving owing to its symplecticity---a crucial property for the reversibility of NUTS.  

To construct an orbit $\calO$, NUTS uses the doubling procedure from Neal's hybrid slice sampler.  Starting from $(\theta,\rho)$, NUTS repeatedly applies the leapfrog integrator to expand the orbit either forward or backward in time with equal probability, doubling its length in physical time with each expansion.  This process continues until a \emph{U-turn} is detected.   Given the constructed orbit $\calO$, the next state $\theta'$ is the position component of a uniform sample from $\calO \cap\slice(y)$.  

Due to the leapfrog algorithm's approximate energy conservation and the uniform selection of the height variable over $[0, e^{-H(\theta,\,\rho)}]$,  most points in the leapfrog orbit remain within the slice when an appropriate leapfrog step size is used.  Importantly, the orbit selection step in NUTS is independent of the slice used for final state selection. This insight led to a key refinement of NUTS, which we discuss next.

\paragraph*{Increasing precision with slice-free NUTS} Recognizing that orbit selection is independent of the slice allows for replacing slice-based state selection with categorical state selection \cite{betancourt2017conceptual}. Instead of selecting the next state uniformly from points within the slice, this approach selects it directly from the leapfrog orbit, assigning each point a probability weight proportional to its density $e^{-H}$. Due to energy errors introduced by the leapfrog integrator, this categorical distribution is not uniform.   This slice-free variant of NUTS is also known as multinomial NUTS \cite{betancourt2017conceptual}. 

Given a countable set $S$ and a discrete density $f: S \to \mathbb{R}_{ \ge 0}$ satisfying $\sum_{x \in S} f(x)< \infty$,
we define the categorical distribution on $S$ weighted by $f$ as:
\begin{equation}\label{eq:cat}
    X\ \sim\ \cat(S,f) \quad\text{if}\quad\mathbb P(X=x)\ \propto\ f(x)\quad\text{for all $x\in S.$}
\end{equation}
In the context of NUTS, this means that each state in $\calO$ is weighted according to the value of $e^{-H}$ at that state.  Compared to slice-based state selection, which selects the next state uniformly from the subset of $\calO$ that lies in the slice, categorical state selection incorporates the precise value of $e^{-H}$ along $\calO$.  This increased precision results in improved performance, as shown in Figure \ref{fig:nounderrun_OS} (b), which shows that selecting the state based on density values rather than a uniform distribution results in more directed motion along the orbit.

\paragraph*{A gradient-free counterpart to slice-free NUTS}  Gradient-free sampling can be achieved by replacing the leapfrog integrator in NUTS with any volume-preserving, gradient-free map.  A natural choice is constant velocity motion, which corresponds to Hamiltonian flow with constant potential energy.  Since the velocity is constant, orbits are fully defined in position space and can be represented as a sequence of points $\calO = (\theta^{\leftmost}, \dots, \theta^{\rightmost})$ lying along lines passing through the current state $\theta$ in the direction of the initial velocity $\rho$.  The U-turn stopping condition of NUTS, designed for curved Hamiltonian trajectories, must be replaced by a stopping condition suitable for constant velocity motion.  

As in slice-free NUTS, the next state $\theta'$ is sampled from $\calO$ according to the categorical distribution, i.e.,  
\begin{equation}\label{eq:catorbit}
\theta' \sim \cat(\calO, \mu)  \quad\text{if}\quad   \mathbb P( \theta' = \tilde{\theta} )\ =\ \frac{\mu(\tilde{\theta})}{\sum_{\breve{\theta} \in \calO} \mu(\breve{\theta})} \quad\text{for all $\tilde{\theta}\in \calO.$}
\end{equation}
Thus, the probability of selecting $\tilde{\theta} \in \calO$ is proportional to $\mu(\tilde{\theta})$, ensuring that the next state is chosen according to the precise target density values along the orbit.

\paragraph*{An implementable Hit-and-Run} Without a stopping condition, the orbit construction of this gradient-free sampler produces an infinite lattice along the line $L(\theta,\rho)=\theta+\mathbb R\,\rho$ through $\theta$ in direction $\rho$.  From this lattice, the subsequent state, as in NUTS, is selected according to the categorical distribution weighted by the target density along the lattice, as in~\eqref{eq:catorbit}.  Interpreting this categorical distribution as an approximation to the restriction of $\mu$ to the line $L(\theta,\rho)$  conceptually connects NURS to Hit-and-Run \cite{Diaconis2007HitandRun,ascolani2024,BoEbOb2024}, which selects its next state from the exact restriction to the line.  Taking the entire line into consideration when selecting its next state enables Hit-and-Run to move globally on the line.

In practice, this infinite-orbit algorithm is not computable, necessitating a restriction to finite lattices via a suitable stopping condition.  The connection to Hit-and-Run inspires the \emph{No-Underrun condition} geared towards keeping the orbit doubling going until the orbit spans most of $\mu$'s restriction to the line.  This ensures that the categorical state selection in NURS closely approximates the exact selection in Hit-and-Run.  

In sum, NURS represents a novel synthesis of ideas from NUTS and Hit-and-Run. It leverages a gradient-free variant of the locally adaptive, orbit-based exploration of NUTS while retaining Hit-and-Run’s ability to make global moves along lines. In this sense, NURS provides a practical implementation of Hit-and-Run.

\paragraph*{Contributions} 

This work introduces NURS, a locally-adaptive, gradient-free MCMC method for sampling from absolutely continuous target distributions. In each iteration, NURS selects an update direction from the unit sphere, thus avoiding artifacts introduced by fixed coordinate directions. The algorithm employs orbit-based exploration, and incorporates a locally adaptive orbit length, where a stopping criterion adjusts the orbit length based on the local scale of the target distribution. To guide state selection, NURS leverages exact evaluations of the target density along the orbit via categorical sampling.

A key theoretical contribution of this work is the interpretation of NURS as a practical approximation of the Hit-and-Run algorithm, a coordinate-free Gibbs sampler that is theoretically appealing but often computationally impractical.  Furthermore, a link between the initial random shift in NURS and Random Walk Metropolis, is established. These connections are not just of theoretical interest---they also inform principled tuning guidelines for the two hyperparameters of NURS: the relationship to Hit-and-Run informs the selection of the threshold parameter used in the No-Underrun condition, while the link to Random Walk Metropolis guides the choice of the lattice spacing.

On the theoretical side, we establish Wasserstein contraction of NURS in Gaussian targets using a carefully designed coupling, extending similar results for Hit-and-Run. This analysis reveals that NURS inherits Hit-and-Run's potential for fast convergence.  We also derive an explicit upper bound on the total variation distance between the transition kernels of NURS and Hit-and-Run, rigorously quantifying their overlap.  This result formally characterizes the extent to which NURS approximates Hit-and-Run by comparing its transition kernel to a piecewise uniform approximation of Hit-and-Run's state selection.

Finally, we conduct a study of NURS in Neal's funnel distribution, a  multi-scale example from  Bayesian hierarchical inference that poses substantial challenges for MCMC methods. To gain theoretical insight, we approximate the local geometry of the funnel at its most and least restrictive scales using Gaussian models. While the funnel does not have a strictly defined smallest scale, the relevant smallest scale can still be meaningfully characterized by focusing on a sufficiently large region where most of the probability mass is concentrated.  Numerical experiments support these findings, demonstrating that a well-tuned NURS can effectively sample from Neal's funnel. 

By providing an implementable Hit-and-Run amenable to theoretical analysis and demonstrating practical effectiveness in a challenging multi-scale target, NURS offers a significant contribution in gradient-free MCMC methods. Our results suggest that NURS is a promising alternative to existing methods, combining strong theoretical foundations with practical applicability in a way that has been largely unexplored.

\paragraph*{Organization of paper}  Section~\ref{sec:nurs} provides a detailed description of NURS.  In Section~\ref{sec:tuning_nurs}, we present quantitative tuning guidelines and connect NURS to both Hit-and-Run and Random Walk Metropolis.   Section~\ref{sec:contr} establishes Wasserstein contraction for NURS in the Gaussian setting (Theorem~\ref{thm:HRcontr}).   Section~\ref{sec:hit_and_run} formalizes the connection between NURS and Hit-and-Run (Theorem~\ref{thm:overlap}). Section~\ref{sec:funnel} theoretically and numerically analyzes NURS in the context of Neal's funnel.  Finally, Section~\ref{sec:reversibility} proves the reversibility of NURS (Theorem~\ref{thm:NURS_reversibility}).   
The paper concludes with a discussion and outlook section (Section~\ref{sec:outlook}), and includes a self-contained pseudocode implementation of NURS (Appendix~\ref{sec:pseudocode}).

\begin{figure}
\centering
\includegraphics[width=\textwidth]{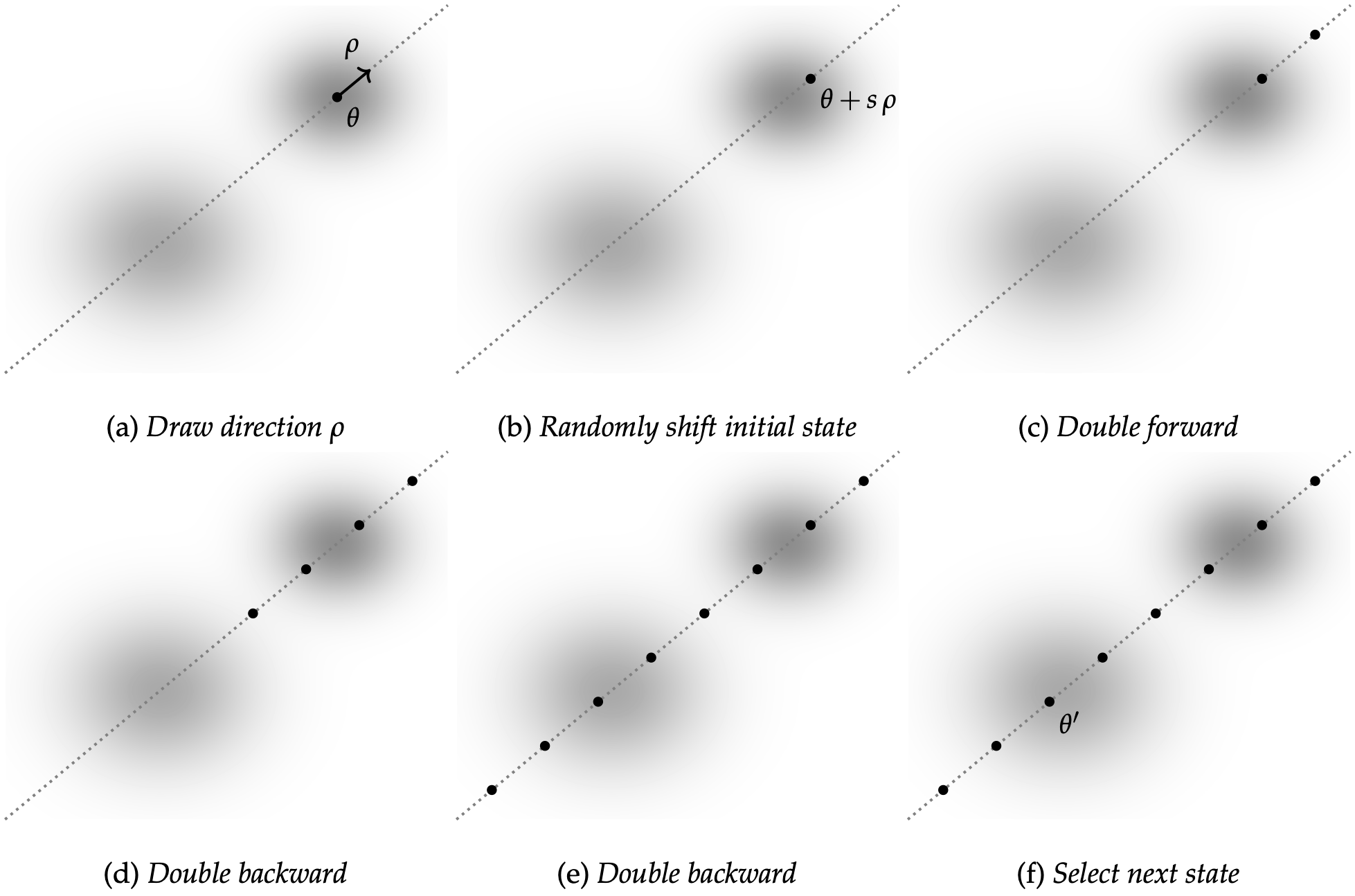}
\caption{\textit{A transition of NURS from an initial state $\theta\in\mathbb R^2$ in the target shown in Figure \ref{fig:slice_sampler}.  NURS first samples a direction $\rho$ from the unit sphere, defining the line along which the orbit will be built.  The initial state is then randomized by applying a Metropolis-adjusted random shift $s\in[-h/2,\,h/2)$ along the line.  Starting from the shifted state, NURS iteratively builds an orbit: a lattice with spacing $h$ on the line. The orbit is doubled at each step either forward or backward with equal probability.  The doubling continues until the No-Underrun condition is met, indicating the orbit spans most of the target restricted to the line.  From the selected orbit, the next state $\theta'$ of NURS is sampled according to the categorical distribution \eqref{eq:catorbit}, with probabilities proportional to the target's density evaluated at the orbit points.}} \label{fig:slice_and_run}
\end{figure}

\section{The No-Underrun Sampler}\label{sec:nurs}

In this section, we provide a detailed description of the No-Underrun Sampler (NURS).
Let $\mu$ denote an absolutely continuous target probability measure on $\mathbb R^d$, with its non-normalized density also denoted by $\mu$.
Starting from the current state $\theta\in\mathbb R^d$, NURS first samples a direction $\rho\sim\tau$ from a probability distribution $\tau$ on the unit sphere $\mathbb S^{d-1}$. Importantly,  the selection of the update direction $\rho$ is coordinate-free, akin to Hit-and-Run.  This direction defines the line through $\theta$ along which the transition is restricted.  

Next, starting from a randomized initial position on the line, NURS generates an orbit in the form of a lattice with spacing $h>0$ on the line.  This orbit is generated iteratively using a doubling procedure, which terminates when a \emph{No-Underrun condition} is met. This condition is designed to ensure that the orbit adequately represents the target distribution restricted to the line.   Once the orbit is selected, the next state is chosen using categorical state selection \eqref{eq:catorbit}, where each point on the orbit is weighted according to its density.

A transition of NURS from current state $\theta \in \mathbb{R}^d$ to next state $\theta' \in \mathbb{R}^d$ is defined in Algorithm~\ref{alg:nurs}. A detailed definition of the orbit selection step follows the algorithm, and Section~\ref{sec:tuning_nurs} discusses the tuning of NURS.  

\begin{algorithm}[ht]
\caption{{\bfseries No-Underrun Sampler}: $\ \theta' \sim\pi_{\NURS}(\theta, \cdot)$}
\begin{description}
\item[Inputs:] $\theta \in \mathbb{R}^d$ (current state); \ \ $h > 0$ (lattice size); \ \ $\epsilon  \ge 0$ (density threshold); \ \ $M \in \mathbb{N}$ (maximum doublings);  \ \ $\tau$ (distribution on the unit sphere $\mathbb{S}^{d-1}$); \ \ $\mu$ (non-normalized target density on $\mathbb{R}^d$)
\vspace*{2pt}
\hrule
\vspace*{6pt}
\item[Hit Step:] Sample direction $\rho \sim \tau$.
\item[Run Step:]\phantom{k}
\begin{itemize}
\item[1.]
Sample and Metropolis-adjust shift:
\begin{itemize}
    \item $s' \sim\Unif([-h/2,\,h/2))$.
    \item $A \sim \operatorname{Bernoulli}\Bigr(1 \wedge \frac{\mu(\theta+s'\rho)}{\mu(\theta)}\Bigr)$.
    \item $s = A \; s'$.
\end{itemize}
\item[2.] Sample orbit $\mathcal{O} \sim \mathrm{Orbit}(\theta+s\rho, \rho, h, \epsilon, M)$.
\item[3.] Sample and return state $\theta' \sim \cat\!\left(\mathcal{O}, \mu \right)$.
\end{itemize}
\end{description}
\label{alg:nurs}
\end{algorithm}

\begin{theorem} \label{thm:NURS_reversibility}
The transition kernel $\pi_{\NURS}$ is reversible with respect to the target $\mu$. 
\end{theorem}

Theorem~\ref{thm:NURS_reversibility} is proved in Section~\ref{sec:reversibility}.

\subsection{Orbit Selection \texorpdfstring{$\mathcal{O} \sim \mathrm{Orbit}(\theta, \rho, h, \epsilon, M)$}{Orbit}}

The orbit selection procedure generates an orbit from which the next state is selected.  The current state is $\theta \in \mathbb{R}^d$, the direction of evolution is $\rho \in \mathbb{S}^{d-1}$, the lattice spacing between points on the orbit is $h > 0$, the density threshold determining when to stop evolution of the orbit is $\epsilon \ge 0,$ and $M \in \mathbb{N}$ is the maximum number of doublings allowed.  

\paragraph*{Elements of orbit construction} Starting from $\theta \in \mathbb{R}^d$, the orbit selection procedure generates a sequence of consecutive states, referred to as an \emph{orbit}:
\[ \calO\ =\ (\theta^{\leftmost}, \dots, \theta^{\rightmost})\ \subset\ \theta+h\,\mathbb Z\,\rho \]
where the orbit lies on the $h$-lattice $\theta + h \, \mathbb{Z} \, \rho$ in the direction $\rho \in \mathbb{R}^d$, starting from $\theta$.  As in Neal's hybrid slice sampler and NUTS, orbits are constructed by repeated doubling.
To double an orbit $\mathcal O$, we first select a direction forward or backward uniformly.  If we are moving forward, the extension is 
$\calO^{\ext} = \mathcal{O} + h\, \left| \mathcal{O} \right| \, \rho$, 
and if backward, 
$\calO^{\ext} = \mathcal{O} - h\, \left| \mathcal{O} \right|\, \rho$.  Here, arithmetic is performed componentwise and $|\calO|$ denotes the number of states in orbit $\calO$.   If we move forward, the doubled orbit is $\mathcal{O} \odot \mathcal{O}^\ext$, and if backward, it is $\mathcal{O}^\ext \odot \mathcal{O}$, where for two orbits $\calO_1=(\theta^{\leftmost}_1, \dots, \theta^{\rightmost}_1)$ and $\calO_2=(\theta^{\leftmost}_2, \dots, \theta^{\rightmost}_2)$, concatenation is defined as
\[ \calO_1 \odot  \calO_2\ =\ (\theta^{\leftmost}_1, \dots, \theta^{\rightmost}_1, \theta^{\leftmost}_2, \dots, \theta^{\rightmost}_2)\;. \]
This doubling procedure continues until one of the following criteria is satisfied: a stopping condition holds, a sub-stopping condition holds, or a predefined maximum number of doublings ($M$) is reached.  The definitions of the stopping and sub-stopping conditions depend on the density threshold $\epsilon$ and are provided next, followed by a detailed description of the orbit selection procedure.

\paragraph*{No-Underrun stopping condition}

Given a lattice spacing $h>0$ and density threshold $\epsilon \ge 0$, an orbit 
$\mathcal{O} = (\theta^{\leftmost}, \dots, \theta^{\rightmost})$ satisfies the No-Underrun \emph{stopping condition} if
\begin{equation}\label{eq:no-Underrun}
\max\bigr(\mu(\theta^{\leftmost}),\,\mu(\theta^{\rightmost}) \bigr) \ \leq \ \epsilon \;h \, \sum_{\breve{\theta} \in \calO} \mu(\breve{\theta})  \;. 
\end{equation}  This condition is scale invariant, i.e., it is unchanged if the target density is rescaled by a constant factor. Since the iterative doubling starts with the single-state orbit $(\theta)$, all resulting orbits satisfy $|\mathcal O| \in2^{\mathbb N}$.  For such orbits, we define the following collection of all sub-orbits obtainable by repeated halving:
\begin{equation}\label{eq:sub-orbits}
 \Bigr\{ \calO_{i,j} \ :\ i \in \bigr\{ 0, 1, \dots, \log_2 |\mathcal O| \bigr\}, \, j \in \bigr\{ 1, 2, \dots, 2^i \bigr\}  \Bigr\}
\end{equation}
where, for each $i \in \{ 0, 1, \dots, \log_2 |\mathcal O| \}$, $\mathcal O_{i,j}$ are defined to be the unique orbits of size $|\mathcal O| 2^{-i}$ such that
\[ \mathcal O\ =\ \mathcal O_{i,1}\odot\mathcal O_{i,2}\odot\cdots\odot\mathcal O_{i,2^i} \;. \]
Finally, an orbit $\mathcal O$ is said to satisfy the \emph{sub-stopping condition} if any of its sub-orbits as defined in \eqref{eq:sub-orbits} satisfies the No-Underrun stopping condition.

\paragraph*{Orbit selection procedure}

Given a state $\theta\in\mathbb R^d$, a direction $\rho\in\mathbb S^{d-1}$, a density threshold $\epsilon \ge 0$, a lattice spacing $h>0$,  and a maximum number of doublings $M \in \mathbb{N}$, the orbit selection procedure in NURS proceeds iteratively as follows, starting with the initial orbit $\calO_0=(\theta)$.  A symmetric Bernoulli process $B \sim \Unif\bigr(\{0,1\}^M\bigr)$ is sampled to determine the direction of each doubling step.  After $k$ doublings ($k \in \{0, \dots, M-1\})$, 
\begin{itemize}
\item For the current orbit $\calO_k$, draw an extension $\calO_k^{\ext} = \begin{cases} \mathcal O_k+h\,|\mathcal O_k|\,\rho & \text{if $B_k=1$,} \\
\mathcal O_k-h\,|\mathcal O_k|\,\rho & \text{if $B_k=0$.} \end{cases}$  
\item If  $\calO_k^{\ext}$ satisfies the sub-stopping condition, select $\calO_k$ as final orbit.  
\item If $\calO_k^{\ext}$ does not satisfy the sub-stopping condition, set $\calO_{k+1}= \begin{cases}  \calO_k\odot \calO_k^{\ext} & \text{if $B_k=1$,} \\ 
\calO_k^{\ext} \odot \calO_k   & \text{if $B_k=0$.} 
\end{cases}$
\item If $\calO_{k+1}$ satisfies the stopping condition or $k+1=M$, select $\calO_{k+1}$ as final orbit.
\item Otherwise, repeat with $\calO_{k+1}$ as the current orbit.
\end{itemize}
This doubling procedure generates a final orbit $\mathcal{O}$ whose size is a power of $2$, i.e., \[
|\calO|=2^\ell \;,  \quad \text{for some} \quad  \ell\in\mathbb N \;, \quad \ell \le M \;.
\]

\subsection{Implementation of NURS}

NURS is fully specified by a threshold $\epsilon \ge 0$, a lattice spacing $h>0$, and a maximum number of doublings $M \in \mathbb{N}$. Tuning guidelines for $\epsilon$ and $h$ are given in Sections~\ref{sec:no-underrun} and~\ref{sec:RWM}. 

Like NUTS \cite{HoGe2014,betancourt2017conceptual}, NURS can be coded so that it does not need to store all $2^M$ states.  Instead of storing all states, the log sum of probabilities is accumulated along with a chosen state.  These are initialized at the initial state and initial log density.  Then as each state is visited, the log sum of probabilities is updated and a probabilistic choice is made as to whether to update the state to the new state.  These are stored along with endpoints of sub-doublings, leading to at most a logarithmic storage cost at the price of an increased number of random number generations. Complete pseudocode for NURS is provided in Appendix~\ref{sec:pseudocode}.

Unlike NUTS, the states within an orbit in NURS are equally spaced along a line, enabling parallel computation for evaluating their log densities. This parallelization makes NURS particularly efficient on modern computing platforms, where we can evaluate the log densities at thousands of states in a single batch, and sample the next state accordingly. Although MCMC is inherently a sequential algorithm, NURS can leverage parallel computation to significantly accelerate runtime.

\section{Quantitative Tuning Guidelines for NURS} 
\label{sec:tuning_nurs}

In this part, we provide  guidelines for tuning the threshold $\epsilon \ge 0$ used in the No-Underrun stopping condition \eqref{eq:no-Underrun} as well as the lattice spacing $h>0$.  These insights are derived by connecting NURS to two other MCMC methods: Hit-and-Run, which is generally not implementable, and Random Walk Metropolis.  By analyzing these connections, we gain insight into the behavior of NURS and derive principled strategies for selecting $\epsilon$ and $h$.

\subsection{Tuning the density threshold \texorpdfstring{$\epsilon$}{epsilon}}\label{sec:no-underrun}

Tuning guidelines for the threshold $\epsilon$ arise from interpreting NURS as an implementable variant of Hit-and-Run. To establish this connection, we first introduce Hit-and-Run and then discretize it to naturally derive an infinite-orbit variant of NURS. The goal of equipping NURS with the key properties of Hit-and-Run motivates the No-Underrun condition, clarifies the correct tuning of $epsilon$, and highlights its implications for the orbit selection procedure.

\paragraph*{Hit-and-Run}
Define the line through $\theta\in \mathbb{R}^d$ in the direction $\rho \in \mathbb S^{d-1}$ as
\[ L(\theta,\rho) = \{  \theta + t \, \rho \mid t \in \mathbb{R} \} \;. \]
It corresponds to the image of the displacement map $\Delta_{\theta,\,\rho}:\mathbb R\to\mathbb R^d$, $\Delta_{\theta,\,\rho}(t)=\theta+t\rho$ which maps a scalar displacement $t$ to the corresponding point on the line.
Once a parametrization of the line is fixed, a regular version of the conditional distribution of $\mu$ given $L(\theta,\rho)$ is the push forward
\[ \mu\bigr(\,\cdot\,\bigr|L(\theta,\rho)\bigr)\ =\ \mu_{\theta,\,\rho}\circ\Delta_{\theta,\,\rho}^{-1} \]
of the one-dimensional displacement measure $\mu_{\theta,\,\rho}$ on $\mathbb R$ defined via its almost everywhere finite density
\begin{equation}\label{eq:sdm}
    \mu_{\theta,\,\rho}(t)\ \propto\ \mu(\theta+t\rho) \;.
\end{equation}

A transition of generalized Hit-and-Run with direction distribution $\tau$ is given in Algorithm~\ref{alg:gH&R}.
Note that this method is generally not implementable because it requires an exact sample from the scalar displacement measure.

\begin{algorithm}[ht]
\caption{Generalized Hit-and-Run: $\ \theta'\sim\pi_{\gHR}(\theta,\cdot)$}\label{algo:HR}
\begin{description}
\item[Inputs:] $\theta \in \mathbb{R}^d$ (current state);   \ \ $\tau$ (direction distribution on the unit sphere $\mathbb{S}^{d-1}$);  \ \ $\mu_{\theta, \rho}$ (displacement measure on $\mathbb{R}$)
\vspace*{2pt}
\hrule
\vspace*{6pt}
\item[Hit Step:] Sample direction $\rho \sim \tau$.
\item[Run Step:] Sample scalar displacement $T\sim\mu_{\theta,\,\rho}$ and return $\theta' = \theta + T \, \rho$.
\end{description}
\label{alg:gH&R}
\end{algorithm}

This is illustrated in Figure \ref{fig:HR_transition_step}.
The original Hit-and-Run method uses the uniform distribution on the unit sphere in the Hit Step while the generalization to arbitrary distributions $\tau$ yields more flexibility that may be used to, e.g., precondition the direction distribution either locally or globally.

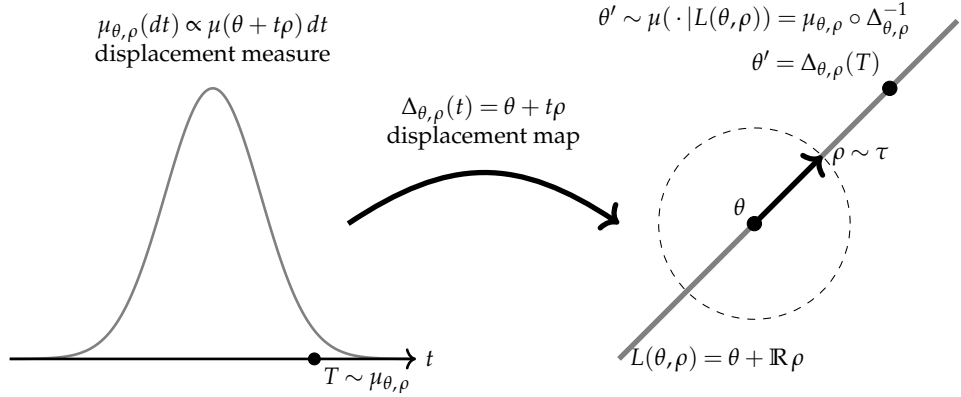
\begin{figure}[t]
\centering
\begin{tikzpicture}[scale=0.9]

\clip (0,0.5) rectangle (16,7);

\draw [gray,line width=1pt,smooth,samples=100,domain=0:6] plot({\x+1},{4*e^(-(\x-3)^2)+1});
\draw[line width=1pt,->] (1,1) -- (7,1) node[at end, right]{$t$};
\node[align=center] at (4,5.7) {$\mu_{\theta,\,\rho}(dt)\propto\mu(\theta+t\rho)\,dt$\\displacement measure};
\filldraw (5.5,1) circle (2.5pt) node[below right]{$T\sim\mu_{\theta,\,\rho}$};

\draw[dashed] (12,3) circle ({sqrt(2)});
\draw[gray, line width=2pt] (10,1) -- (15,6) node[black,at start, right] {$L(\theta,\rho)=\theta+\mathbb R\,\rho$};
\filldraw (12,3) circle (3pt) node[above left]{$\theta$};
\draw[line width=2pt, ->] (12, 3) -- (13, 4) node[pos=1, right] {$\rho\sim\tau$};
\filldraw (14,5) circle (3pt) node[above left]{$\theta'=\Delta_{\theta,\,\rho}(T)$};
\node[align=center] at (12,6) {$\theta'\sim\mu(\,\cdot\,|L(\theta,\rho))=\mu_{\theta,\,\rho}\circ\Delta_{\theta,\,\rho}^{-1}$};

\draw[line width=2pt, ->] (6,3) .. controls (7.5,4) and (8.5,4) .. (10,3);
\node[align=center] at (8,4.5) {$\Delta_{\theta,\,\rho}(t)=\theta+t\rho$\\displacement map};

\end{tikzpicture}

\caption{\textit{
Transition step of generalized Hit-and-Run from state $\theta$.
}}
\label{fig:HR_transition_step}
\end{figure}

\paragraph*{Infinite-orbit NURS} To turn Hit-and-Run into an implementable algorithm, the exact sampling from the scalar displacement measure needs to be replaced by computationally feasible steps.  By discretizing the line $L(\theta,\rho)$ into an infinite lattice and selecting the subsequent state from a categorical approximation of the scalar displacement measure over this lattice, we obtain a discretization of generalized Hit-and-Run.  Let $h>0$ be the spacing of the lattice. A transition of this method, referred to as \emph{uniform-shifted infinite-orbit NURS}, is given in Algorithm~\ref{alg:infty_unadjusted}.

\begin{algorithm}[ht]
\caption{Uniform-shifted infinite-orbit NURS: $\ \theta'\sim\pi_{\infty,u}(\theta,\cdot)$}
\begin{description}
\item[Inputs:] $\theta \in \mathbb{R}^d$ (current state); \ \ $h > 0$ (lattice size);   \ \ $\tau$ (direction distribution on the unit sphere $\mathbb{S}^{d-1}$);  \ \ $\mu_{\theta, \rho}$ (displacement density on $\mathbb{R}$)
\vspace*{2pt}
\hrule
\vspace*{6pt}
\item[Hit Step:] Sample direction $\rho \sim \tau$.
\item[Run Step:]\phantom{k}
\begin{itemize}
\item[1.] Sample a random shift $s\sim\Unif([-h/2,\,h/2))$.
\item[2.] Sample a scalar displacement $T\sim \cat\bigr(h \mathbb{Z}, \mu_{\theta+s\rho,\,\rho}\bigr)$ and return $\theta' = \theta + (s+T) \, \rho$.
\end{itemize}
\end{description}
\label{alg:infty_unadjusted}
\end{algorithm}

Due to the shift $s$, the algorithm uniformly selects a lattice on $L(\theta, \rho)$ with fixed spacing $h$.  However, this uniform shift breaks reversibility of the Markov chain.  To restore reversibility, we add a Metropolis filter, which yields \emph{infinite-orbit NURS} given in Algorithm~\ref{alg:infty}.

\begin{algorithm}[ht]
\caption{Infinite-orbit NURS: $\ \theta'\sim\pi_\infty(\theta,\cdot)$}
\begin{description}
\item[Inputs:] $\theta \in \mathbb{R}^d$ (current state); \ \ $h > 0$ (lattice size);   \ \ $\tau$ (direction distribution on the unit sphere $\mathbb{S}^{d-1}$);  \ \ $\mu_{\theta, \rho}$ (displacement density on $\mathbb{R}$)
\vspace*{2pt}
\hrule
\vspace*{6pt}
\item[Hit Step:] Sample direction $\rho \sim \tau$.
\item[Run Step:]\phantom{k}
\begin{itemize}
\item[1.]
Sample and Metropolis-adjust shift:
\begin{itemize}
    \item $s' \sim\Unif([-h/2,\,h/2))$.
    \item $A \sim \operatorname{Bernoulli}\Bigr(1 \wedge \frac{\mu_{\theta,\,\rho}(s')}{\mu_{\theta,\,\rho}(0)}\Bigr)$.
    \item $s = A \; s'$.
\end{itemize}
\item[2.] Sample scalar displacement $T\sim \cat\bigr(h \mathbb Z, \mu_{\theta+s\rho,\,\rho}\bigr)$ and return $\theta' = \theta + (s+T) \, \rho$.
\end{itemize}
\end{description}
\label{alg:infty}
\end{algorithm}

It is worth emphasizing that, unlike many other sampling methods involving Metropolis filters, such as Random Walk Metropolis (Algorithm~\ref{alg:RWM}), Infinite-orbit NURS does not necessarily remain in its initial state in case of rejection.  Instead, it proceeds with the categorical state selection on an unshifted lattice.

Infinite-orbit NURS is obtained from NURS (Algorithm~\ref{alg:nurs}) by taking a zero threshold $\epsilon=0$ and allowing an infinite number of doublings $M=\infty$.  Indeed, then the orbit selection of Algorithm~\ref{alg:nurs} produces
\[ \calO \ =\ \theta+(s+h\mathbb Z)\,\rho \]
from which the subsequent state is selected according to
\[ \theta'\ \sim\ \cat\bigr(\calO, \mu\bigr)\;. \]
This corresponds precisely to step 2 of the Run Step in Algorithm~\ref{alg:infty}.

\paragraph*{The No-Underrun condition}  A key feature of Hit-and-Run is its remarkable ability to move to any point along a line passing through its current state by sampling from the target distribution restricted to that line. In flat regions, such as the mouth of Neal's funnel (see Section \ref{sec:funnel}), this property enables Hit-and-Run to make large moves, resulting in fast mixing. Infinite-orbit NURS, introduced in the previous section, shares this desirable property: Its orbit spans the entire line and, up to a lattice discretization, accounts for the entire mass of the target distribution on the line. However, the infinite length of such orbits makes this approach impractical to implement.  

To address this limitation, NURS employs a stopping condition specifically designed to preserve the essential characteristic of its infinite-orbit limit while remaining computationally feasible.
The stopping condition ensures that a finite-length orbit is selected that contains the bulk of the distribution along the line.
This motivates the \emph{No-Underrun condition}, defined  in \eqref{eq:no-Underrun}.  

\paragraph*{Threshold tuning}

To better understand \eqref{eq:no-Underrun} and how to properly tune the threshold $\epsilon \ge 0 $, consider a continuous orbit corresponding to a segment of the line $L(\theta,\rho)$.  The No-Underrun condition for this continuous orbit can be expressed in terms of $\mu_{\theta,\,\rho}$, the restriction of the target measure $\mu$ to the line. Let $a,b\in\mathbb R$ represent the scalar displacements corresponding to the endpoints of the continuous orbit, i.e., $\Delta_{\theta,\,\rho}(a)=\theta^\leftmost$ and $\Delta_{\theta,\,\rho}(b)=\theta^\rightmost$.
Then condition \eqref{eq:no-Underrun} approximates
the following continuous version \begin{equation}\label{eq:no-Underrun_cont}
\max\bigr(\mu_{\theta,\,\rho}(a),\,\mu_{\theta,\,\rho}(b) \bigr) \ \leq \ \epsilon \;\mu_{\theta,\,\rho}\bigr((a,b)\bigr) 
\end{equation}
in the sense that the right hand side of \eqref{eq:no-Underrun} is a Riemann sum approximation of the right hand side of \eqref{eq:no-Underrun_cont}.

Now suppose that $\mu_{\theta,\,\rho}$ has exponential tails and $\epsilon$ is tuned such that barriers exceed the threshold $\epsilon$  (see Figure \ref{fig:nounderrun}).
Under these assumptions, \eqref{eq:no-Underrun_cont} is satisfied only for intervals $(a,b)$ that contain the bulk of the distribution, while excluding tail regions encompassing total probability mass of order $\epsilon$.  Condition \eqref{eq:no-Underrun_cont} is therefore well suited for determining whether a finite-length orbit contains the bulk of the distribution.  For optimal results, it is recommended to set $\epsilon$ significantly smaller than expected barriers.

\begin{figure}[ht]
\centering
\begin{tikzpicture}[scale=1]
\clip (-2,-.5) rectangle (12,3);
\node[right] at ({8},{3*(e^(-(8-3)^2/2)+e^(-(8-7)^2/2)}) {$\mu_{\theta,\,\rho}$};
\fill [NavyBlue!40,samples=100,domain=1.11:8.89] plot ({\x}, {3*(e^(-(\x-3)^2/2)+e^(-(\x-7)^2/2)}) |- (1.11,0);
\fill [gray!40,samples=100,domain=0:1.11] plot ({\x}, {3*(e^(-(\x-3)^2/2)+e^(-(\x-7)^2/2)}) |- (0,0);
\fill [gray!40,samples=100,domain=8.89:10] plot ({\x}, {3*(e^(-(\x-3)^2/2)+e^(-(\x-7)^2/2)}) |- (8.89,0);
\draw [dashed,smooth,samples=100,domain=0:10] plot ({\x},{3*(e^(-(\x-3)^2/2)+e^(-(\x-7)^2/2)});
\draw (0,0) -- (10,0);
\draw (0,.5) -- (10,.5) node[pos=0,left] {$\epsilon$};
\draw (1.11,0) -- (1.11,0.5);
\draw (8.89,0) -- (8.89,0.5);
\draw[line width=1pt] (1.11,0) -- (8.89,0);
\draw[line width=1pt] (1.11,-0.1) -- (1.11,0.1) node[pos=0,below] {$a$};
\draw[line width=1pt] (8.89,-0.1) -- (8.89,0.1) node[pos=0,below] {$b$};
\end{tikzpicture}
\caption{\textit{The No-Underrun condition \eqref{eq:no-Underrun_cont} ensures that for exponentially-tailed distributions with barriers exceeding $\epsilon$, the interval $(a,b)$ contains the bulk of the distribution, with the total probability mass outside $(a,b)$ being $O(\epsilon)$.}}
\label{fig:nounderrun}
\end{figure}
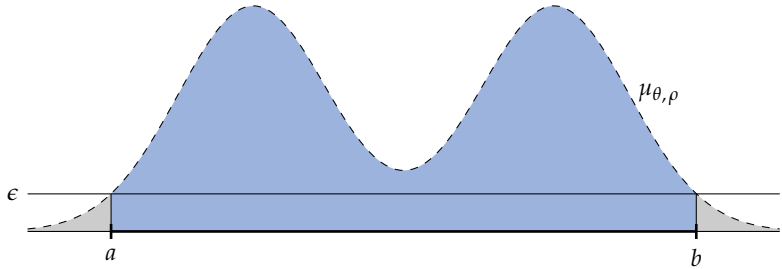

\paragraph*{Orbit selection with No-Underrun}

We now detail how the orbit selection with the No-Underrun stopping condition operates in the setting of exponentially-tailed distributions when the threshold $\epsilon$ is properly tuned.  Starting from a state $\theta$ within the bulk of the distribution restricted to the line (see Figure \ref{fig:nounderrun_OS} (a)), the procedure iteratively doubles the orbit until it contains the bulk up to the threshold $\epsilon$.
At each step, extensions $\mathcal O^{\ext}$ are considered, but none satisfy  the No-Underrun condition until the final finite-length orbit $\mathcal O$ contains the bulk of the distribution.

If the chain starts in the tails (see Figure \ref{fig:nounderrun_OS} (b)),  an extension $\mathcal O^{\ext}$ may satisfy the No-Underrun condition on its own.
In such cases, the procedure selects the orbit $\mathcal O$ prior to the proposed extension, which spans only a portion of the tails.
However, because the state selection weights decay exponentially fast away from the bulk, NURS makes a directed move towards the bulk with high probability.

\begin{figure}[ht]
\centering
\begin{minipage}{0.48\textwidth}
\centering
\begin{tikzpicture}[scale=.8]
\draw [dashed,smooth,samples=100,domain=0:8] plot ({\x},{3*e^(-(\x-4)^2/0.3)});
\node[right] at ({4.5},{3*e^(-(5-4)^2/3)}) {$\mu_{\theta,\,\rho}$};
\draw (0,0) -- (8,0);
\filldraw (3.7,0) circle (2pt) node[above] {$\theta$};

\draw (3.5,-0.25) -- (4,-0.25);
\draw (3.5,-0.2) -- (3.5,-0.3);
\draw (4,-0.2) -- (4,-0.3);
\filldraw ({3.5},{3*e^(-(3.5-4)^2/0.3)}) circle (1pt);
\draw[dotted] (4,-0.2) -- ({4},{3*e^(-(4-4)^2/0.3)});
\filldraw ({4},{3*e^(-(4-4)^2/0.3)}) circle (1pt);

\draw (3.5,-0.5) -- (4.5,-0.5);
\draw (3.5,-0.45) -- (3.5,-0.55);
\draw (4.5,-0.45) -- (4.5,-0.55);
\draw[dotted] (3.5,-0.5) -- ({3.5},{3*e^(-(3.5-4)^2/0.3)});
\filldraw ({3.5},{3*e^(-(3.5-4)^2/0.3)}) circle (1pt);
\draw[dotted] (4.5,-0.75) -- ({4.5},{3*e^(-(4.5-4)^2/0.3)});
\filldraw ({4.5},{3*e^(-(4.5-4)^2/0.3)}) circle (1pt);

\draw (2.5,-0.75) -- (4.5,-0.75);
\draw (2.5,-0.7) -- (2.5,-0.8);
\draw (4.5,-0.7) -- (4.5,-0.8);
\draw[dotted] (2.5,-1) -- ({2.5},{3*e^(-(2.5-4)^2/0.3)});
\filldraw ({2.5},{3*e^(-(2.5-4)^2/0.3)}) circle (1pt);

\draw (2.5,-1) -- (6.5,-1) node[pos=.5,below] {$\mathcal O$};
\draw (2.5,-0.95) -- (2.5,-1.05);
\draw (6.5,-0.95) -- (6.5,-1.05);
\draw[dotted] (6.5,-1) -- ({6.5},{3*e^(-(6.5-4)^2/0.3)});
\filldraw ({6.5},{3*e^(-(6.5-4)^2/0.3)}) circle (1pt);
\end{tikzpicture}
\caption*{(a)}
\end{minipage}
\begin{minipage}{0.48\textwidth}
\centering
\begin{tikzpicture}[scale=.8]
\draw [dashed,smooth,samples=100,domain=0:8] plot ({\x},{3*e^(-(\x-6)^2/0.3)});
\node[right] at ({6.5},{3*e^(-(5-6)^2/3)}) {$\mu_{\theta,\,\rho}$};
\draw (0,0) -- (8,0);
\filldraw (1.2,0) circle (2pt) node[above] {$\theta$};
\filldraw (4,0) circle (2pt) node[above] {$\theta'$};

\draw (1,-0.25) -- (1.5,-0.25);
\draw (1,-0.2) -- (1,-0.3);
\draw (1.5,-0.2) -- (1.5,-0.3);
\filldraw ({1},{3*e^(-(1-6)^2/0.3)}) circle (1pt);
\draw[dotted] (1.5,-0.2) -- ({1.5},{3*e^(-(1.5-6)^2/0.3)});
\filldraw ({1.5},{3*e^(-(1.5-6)^2/0.3)}) circle (1pt);

\draw (1,-0.5) -- (2,-0.5);
\draw (1,-0.45) -- (1,-0.55);
\draw (2,-0.45) -- (2,-0.55);
\draw[dotted] (1,-0.5) -- ({1},{3*e^(-(1-6)^2/0.3)});
\filldraw ({1},{3*e^(-(1-6)^2/0.3)}) circle (1pt);
\draw[dotted] (2,-0.75) -- ({2},{3*e^(-(2-6)^2/0.3)});
\filldraw ({2},{3*e^(-(2-6)^2/0.3)}) circle (1pt);

\draw (0,-0.75) -- (2,-0.75);
\draw (0,-0.7) -- (0,-0.8);
\draw (2,-0.7) -- (2,-0.8);
\draw[dotted] (0,-1) -- ({0},{3*e^(-(0-6)^2/0.3)});
\filldraw ({0},{3*e^(-(0-6)^2/0.3)}) circle (1pt);

\draw (0,-1) -- (4,-1) node[pos=.5,below] {$\mathcal O$};
\draw (0,-0.95) -- (0,-1.05);
\draw (4,-0.95) -- (4,-1.05);
\draw[dotted] (4,-1) -- ({4},{3*e^(-(4-6)^2/0.3)});
\filldraw ({4},{3*e^(-(4-6)^2/0.3)}) circle (1pt);

\draw[dashed] (4,-1) -- (8,-1) node[pos=.5,below] {$\mathcal O^{\ext}$};
\draw (8,-0.95) -- (8,-1.05);
\draw[dotted] (8,-1) -- ({8},{3*e^(-(8-6)^2/0.3)});
\filldraw ({8},{3*e^(-(8-6)^2/0.3)}) circle (1pt);
\end{tikzpicture}
\caption*{(b)}
\end{minipage}
\caption{\textit{No-Underrun based orbit selection from the bulk and the tails.}}
\label{fig:nounderrun_OS}
\end{figure}

\subsection{Tuning the lattice spacing \texorpdfstring{$h$}{h}}\label{sec:RWM}

Tuning guidance for the lattice spacing $h$ results from a similarity between the initial shift in NURS and the Random Walk Metropolis method.  We first explain this connection and subsequently provide tuning guidelines implied by a rigorous result on the interplay between $h$ and the Metropolis accept probability in NURS.

The randomization of the initial point in NURS is a fundamental aspect of the method, ensuring that the orbit lies on a lattice randomly selected from the set of all lattices with spacing $h$ along $L(\theta,\rho)$.  It serves two key purposes.  First, it regularizes the transition kernel, which would otherwise be supported on a union of spheres centered at the current state making it singular with respect to Lebesgue measure.  This simplifies theoretical analyses of NURS' transitions (see Sections~\ref{sec:contr} and~\ref{sec:hit_and_run}).  Second, it establishes a connection between NURS and RWM, with important implications for its performance and tuning.

\paragraph*{Connection to Random Walk Metropolis}

When the lattice spacing is small relative to the scale of the target $\mu$ at the current state $\theta$ in direction $\rho$, NURS operates in the \emph{Hit-and-Run regime} behaving similarly to Hit-and-Run (see Figure \ref{fig:regimes} (a)).  However, when the lattice spacing $h$ is large compared to this local scale, the weights in the categorical distribution used to select the next state become negligible for all but one dominant lattice point (see Figure \ref{fig:regimes} (b)).  As a result, the state selection in NURS becomes trivial, as the dominant lattice point is chosen as the next state with high probability.  In this case, the transition is dominated by the Metropolis-adjusted shift.  This behavior corresponds to the \emph{random walk regime}, where NURS behaves similarly to RWM, a transition of which is given in Algorithm~\ref{alg:RWM}.

\begin{algorithm}[ht]
\caption{Random Walk Metropolis: $\ \theta'\sim\pi_{\RWM}(\theta,\cdot)$}
\begin{description}
\item[Inputs:] $\theta \in \mathbb{R}^d$ (current state); \ \ $h > 0$ (lattice size);   \ \ $\mu$ (target density on $\mathbb{R}^{d}$)
\vspace*{2pt}
\hrule
\vspace*{6pt}
\item[Proposal Step:] Sample increment $\xi\sim\mathcal N\bigr(0,\,h^2d^{-1}I_d\bigr)$.
\item[Metropolis Step:] Sample $A \sim \operatorname{Bernoulli}\Bigr(1\wedge \frac{\mu(\theta+\xi)}{\mu(\theta)}\Bigr)$ and return $\theta'=\theta+A\xi$.
\end{description}
\label{alg:RWM}
\end{algorithm}

Although normally distributed increments are common in the Proposal Step, other symmetric distributions can also be used.  The crucial point is the increment distribution's scale relative to the local scale of $\mu$, as this determines the Metropolis accept probability.  As the covariance of the increments, $h^2d^{-1}I_d$, is chosen to match the scale of the random shift increments of NURS in the random walk regime, there is a clear analogy between the Metropolis accept probabilities in Algorithms~\ref{alg:nurs} and~\ref{alg:RWM} via~\eqref{eq:sdm}.

\begin{figure}[ht]
\centering
\begin{minipage}{0.48\textwidth}
\begin{tikzpicture}[scale=.8]
\clip (0,-.5) rectangle (8,4.5);
\fill[gray!50] (2.75,-.1) rectangle (4.7,.1);
\draw [dashed,smooth,samples=100,domain=0:8] plot ({\x},{3*e^(-(\x-4)^2/5)});
\node[right] at ({5},{4*e^(-(5-4)^2/3)}) {$\mu_{\theta+s\rho,\,\rho}$};
\filldraw (3.7,0) circle (2pt) node[below] {$\theta$};
\draw (0,0) -- (8,0);
\foreach \k in {1,...,5} {
\draw[line width=1pt] (1.5*\k-0.3,-.1) -- (1.5*\k-0.3,.1);
\draw ({1.5*\k-0.3},0) -- ({1.5*\k-0.3},{3*e^(-(1.5*\k-0.3-4)^2/5)});
\filldraw ({1.5*\k-0.3},{3*e^(-(1.5*\k-0.3-4)^2/5)}) circle (1.5pt);
}
\end{tikzpicture}
\caption*{(a) \textit{NURS in the Hit-and-Run regime}}
\end{minipage}
\begin{minipage}{0.48\textwidth}
\begin{tikzpicture}[scale=.8]
\clip (0,-.5) rectangle (8,4.5);
\fill[gray!50] (2.75,-.1) rectangle (4.7,.1);
\draw [dashed,smooth,samples=100,domain=0:8] plot ({\x},{4*e^(-(\x-4)^2/0.3)});
\node[right] at ({4.5},{4*e^(-(5-4)^2/3)}) {$\mu_{\theta+s\rho,\,\rho}$};
\filldraw (3.7,0) circle (2pt) node[below] {$\theta$};
\draw (4.2,0) -- (4.2,{4*e^(-(4.2-4)^2/0.3)});
\filldraw (4.2,{4*e^(-(4.2-4)^2/0.3)}) circle (1.5pt);
\draw (0,0) -- (8,0);
\foreach \k in {1,...,5} {
\draw[line width=1pt] (1.5*\k-0.3,-.1) -- (1.5*\k-0.3,.1);}
\end{tikzpicture}
\caption*{(b) \textit{NURS in the random walk regime}}
\end{minipage}
\caption{\textit{NURS connecting to Hit-and-Run and Random Walk Metropolis depending on the relation between lattice spacing and scale of the target along the line.  The gray intervals represent all possible  shifts of the initial state $\theta$.}}
\label{fig:regimes}
\end{figure}

\paragraph*{The Metropolis step in NURS}

As long as the Metropolis step accepts at least with constant probability, RWM explores the state space similarly to the underlying random walk and therefore inherits its diffusive mixing properties.  However, if the acceptance probability degenerates, the chain can get trapped, possibly leading to exponentially slower mixing.  To avoid such bottlenecks, controlling the acceptance probability is critical.

This is addressed elegantly in \cite{andrieuAAP}, where the authors assume the following regularity condition on the target distribution $\mu(d\theta)\propto\exp(-U(\theta))\,d\theta$ where $U:\mathbb R^d\to\mathbb R$:
\begin{equation}\label{eq:Ureg}
    U(\theta')-U(\theta)-(\theta'-\theta)\cdot\nabla U(\theta)\ \leq\ \psi\bigr(|\theta'-\theta|\bigr)\quad\text{for all $\theta,\theta'\in\mathbb R^d$,}
\end{equation}
for some non-decreasing $\psi:[0,\infty)\to[0,\infty)$.  Under this condition, they show that the acceptance probability of RWM (Algorithm \ref{alg:RWM}) is globally lower-bounded:
\begin{equation}\label{eq:RWMreject}
    \mathbb E_{\xi}\Bigr(1\wedge\frac{\mu(\theta+\xi)}{\mu(\theta)}\Bigr)\ \geq\ \frac12\exp\bigr(-\mathbb E_\xi\psi\bigr(|\xi|\bigr)\bigr)\quad\text{for all $\theta\in\mathbb R^d$.}
\end{equation}

An important instance of $\psi$ is  $\psi(r)=\frac12Lr^2$ for some $L>0$, in which case \eqref{eq:Ureg} corresponds to the commonly used notion of \emph{$L$-smoothness}.
For this $\psi$, the expectation becomes $\mathbb E_\xi\psi(|\xi|)=\frac12Lh^2$ so that the acceptance probability is lower-bounded by an absolute constant for $h=O(L^{-1/2})$.
This choice of $h$ ensures that the proposed increments are adapted to the most restrictive scale relevant for effective sampling, specifically within regions containing sufficient probability mass. For example, in Neal's funnel, the neck region can become arbitrarily narrow, making naive step sizes ineffective. However, by considering only the portion of the neck where a significant fraction of the probability mass lies, $h$ can be appropriately scaled to maintain efficient exploration (see Section \ref{sec:funnel}).

Inspired by the connection between NURS and RWM,  similar control over the Metropolis step of NURS can be established by following \cite{andrieuAAP}.
For a given $\theta\in\mathbb R^d$ and $\rho\in\mathbb S^{d-1}$, the acceptance probability in NURS (Algorithm \ref{alg:infty}) is defined as
\begin{equation}\label{eq:alpha}
    \alpha(\theta,\rho)\ =\ \mathbb E_{s'}\Bigr(1\wedge\frac{\mu_{\theta,\,\rho}(s')}{\mu_{\theta,\,\rho}(0)}\Bigr) \quad\text{where}\quad s'\ \sim\ \Unif\bigr([-h/2,\,h/2)\bigr)\;.
\end{equation}

The following lemma establishes a global lower bound on the acceptance probability under the regularity condition~\eqref{eq:Ureg}.   For potentials whose gradient is only locally Lipschitz continuous (e.g., Neal's funnel), localization is used by picking $h$ to obtain a lower bound in a sufficiently large region where most of the probability mass is concentrated.

\begin{lemma}\label{lem:accprob} Let $h>0$ be the lattice spacing.  Suppose \eqref{eq:Ureg} holds for some non-decreasing $\psi:[0,\infty)\to[0,\infty)$.  Then,
\[ \alpha(\theta,\rho)\ \geq\ \frac12\exp\bigr(-\psi(h/2)\bigr)\quad\text{for all $\theta\in\mathbb R^d$ and $\rho\in\mathbb S^{d-1}$.} \]
\end{lemma}

The analogy with RWM provides guidance for tuning the lattice spacing of NURS.  Specifically, $h$ should be chosen similarly to RWM, ensuring that  bottlenecks due to degenerate acceptance probabilities are avoided.

\begin{proof}
Let $\theta\in\mathbb R^d$ and $\rho\in\mathbb S^{d-1}$.
For all $s'\in[-h/2,\,h/2)$, inserting the definition \eqref{eq:sdm} of the scalar displacement measure, \eqref{eq:Ureg}, applying $1\wedge(\mathsf{ab})\geq(1\wedge\mathsf a)(1\wedge\mathsf b)$ for $\mathsf a,\mathsf b\geq0$, and using that $\psi$ is non-decreasing shows
\begin{align*}
1\wedge\frac{\mu_{\theta,\,\rho}(s')}{\mu_{\theta,\,\rho}(0)}
\ &=\ 1\wedge e^{-(U(\theta+s'\rho)-U(\theta))}
\ \geq\ 1\wedge e^{-s'\rho\cdot\nabla U(\theta)-\psi(|s'|)} \\
\ &\geq\ \bigr(1\wedge e^{-s'\rho\cdot\nabla U(\theta)}\bigr)\bigr(1\wedge e^{-\psi(|s'|)}\bigr) \ \geq\ \bigr(1\wedge e^{-s'\rho\cdot\nabla U(\theta)}\bigr)e^{-\psi(h/2)}\;.
\end{align*}
By \eqref{eq:alpha}, integrating over $s'\sim\Unif\bigr([-h/2,\,h/2)\bigr)$ then yields
\[ \alpha(\theta,\rho)\ \geq\ \mathbb E_{s'}\bigr(1\wedge e^{-s'\rho\cdot\nabla U(\theta)}\bigr)e^{-\psi(h/2)}\ =\ \mathbb E_{s'}\bigr(1\wedge e^{+s'\rho\cdot\nabla U(\theta)}\bigr)e^{-\psi(h/2)}\;, \]
where we changed variables to rewrite the lower bound.  
Averaging these expressions and using $(1\wedge\mathsf a)+(1\wedge\mathsf a^{-1})\geq1$ for $\mathsf a\geq0$ finishes the proof:
\[ \alpha(\theta,\rho)\ \geq\ \frac12\mathbb E_{s'}\Bigr(\bigr(1\wedge e^{-s'\rho\cdot\nabla U(\theta)}\bigr)+\bigr(1\wedge e^{+s'\rho\cdot\nabla U(\theta)}\bigr)\Bigr)e^{-\psi(h/2)}\ \geq\ \frac12e^{-\psi(h/2)}\;. \] \end{proof}

\section{Wasserstein contraction of NURS}\label{sec:contr}

In this section, we build on the previously established connection between NURS and generalized Hit-and-Run to extend recent coupling-based contraction results from Hit-and-Run, as presented in \cite{BoEbOb2024}, to NURS. Specifically, we analyze the case where the target distribution is a centered multivariate Gaussian and show that the uniform-shifted infinite-orbit NURS (Algorithm~\ref{alg:infty_unadjusted}) exhibits contraction in Wasserstein distance.   More generally, the study of locally adaptive MCMC methods in the Gaussian setting remains an open and important problem. While Gaussian targets are analytically tractable, the transition kernel of NURS --- like that of Hit-and-Run --- is not Gaussian, making its analysis nontrivial. 

The argument does not directly extend to infinite-orbit NURS with Metropolis-adjusted shift, as the coupling construction involves careful alignment of the coupled transitions by suitably coupling their underlying shifts.  With the Metropolis filter, the probability of selecting an unshifted orbit is positive preventing alignment.  This is consistent with the general observation that Metropolis adjustment often complicates Wasserstein contraction \cite{BoEbZi2020}.   A recent coupling framework for the mixing analysis of Metropolis-adjusted Markov chains addresses this issue by only requiring Wasserstein contraction of the unadjusted chain, among other assumptions \cite{BouRabeeOberdoerster2023}.

Consider the centered Gaussian target measure
\[ \gamma^{\C}\ =\ \mathcal{N}(0,\mathcal{C})\quad\text{on $\mathbb R^d$} \]
where $\C$ is a symmetric positive definite covariance matrix.

The $L^2$-Wasserstein distance between two probability measures $\eta$ and $\nu$ on $\mathbb{R}^d$, with respect to the metric induced by the norm  $|\theta|_{\C^{-1/2}}=|\C^{-1/2}\theta|$, is defined as
\begin{equation}\label{eq:Wdef}
    \W^2_{\C^{-1/2}}(\eta,\nu)\ =\ \inf\bigr(\mathbb E|\theta-\tilde\theta|^2_{\C^{-1/2}}\bigr)^{1/2}
\end{equation}
where the infimum is taken over all couplings of $\eta$ and $\nu$, i.e., all pairs of random variables $(\theta,\tilde\theta)$ defined on a common probability space such that $\theta\sim\eta$ and $\tilde\theta\sim\nu$.  Here, for any symmetric positive definite matrix $A$, the matrix $A^{1/2}$ denotes its principal square root.

The following theorem extends the contraction for generalized Hit-and-Run, recently established in \cite{BoEbOb2024}, to the setting of uniform-shifted infinite-orbit NURS.

\begin{theorem}\label{thm:HRcontr}
For $\mu=\gamma^\C$, uniform-shifted infinite-orbit NURS with transition kernel $\pi_{\infty,u}$ as defined in Algorithm \ref{alg:infty_unadjusted} satisfies
\begin{equation}\label{eq:Wcontr}
    \W^2_{\C^{-1/2}}\bigr(\pi_{\infty,u}(\theta,\cdot),\, \pi_{\infty,u}(\tilde \theta,\cdot)\bigr)\ \leq\ (1-\lambda)|\theta-\tilde \theta|_{\C^{-1/2}}\quad\text{for all $\theta,\tilde \theta\in\mathbb R^d$}
\end{equation}
with contraction rate given by 
\begin{align} \lambda\ =\ \frac12\inf_{|\zeta|=1}\mathbb E_{\rho\sim\tau}\Bigr(\zeta\cdot\frac{\C^{-1/2}\rho}{|\C^{-1/2}\rho|}\Bigr)^2\;. \label{eq:rate}
\end{align}
\end{theorem}

The theorem provides a lower bound for the global Wasserstein contraction rate, defined as the supremum over all $\lambda$ satisfying \eqref{eq:Wcontr}.
This implies a lower bound for the coarse Ricci curvature of $(\mathbb R^d,|\cdot|_{\C^{-1/2}})$ equipped with uniform-shift infinite-orbit NURS for $\mu=\gamma^\C$ which has far-reaching consequences \cite{OLLIVIER2009,JOULIN2010}.

To prove the theorem, we construct an explicit coupling $(\theta,\tilde\theta)$ of $\pi_{\infty,u}(\theta,\cdot)$ and $\pi_{\infty,u}(\tilde\theta,\cdot)$ that satisfies
\[ \bigr(\mathbb E|\theta-\tilde\theta|^2_{\C^{-1/2}}\bigr)^{1/2}\ \leq\ (1-\lambda)|\theta-\tilde\theta|_{\C^{-1/2}} \]
as the left hand side upper bounds the Wasserstein distance on the right hand side of \eqref{eq:Wcontr}, cf. \eqref{eq:Wdef}.
We first discuss the corresponding coupling construction for generalized Hit-and-Run as presented in \cite{BoEbOb2024} and subsequently extend it to uniform-shifted infinite-orbit NURS.

For $\theta\in\mathbb R^d$ and $\rho\in\mathbb S^{d-1}$, the scalar displacement measure corresponding to $\gamma^\C$ takes the form
\begin{equation}\label{eq:sdmGaussian}
    \gamma^\C_{\theta,\,\rho}\ =\ \mathcal N\Bigr(-\frac{\theta\cdot\C^{-1}\rho}{|\C^{-1/2}\rho|^2},\,|\C^{-1/2}\rho|^{-2}\Bigr)\;.
\end{equation}

The contractive coupling for generalized Hit-and-Run (Algorithm \ref{alg:gH&R}) is based on the observation that writing the scalar displacement used in the Run Step as
\begin{equation}\label{eq:TgHR}
    T\ =\ -\frac{\theta\cdot\C^{-1}\rho}{|\C^{-1/2}\rho|^2}\ +\ Z\quad\text{with}\quad Z\ \sim\ \mathcal N\bigr(0,\,|\C^{-1/2}\rho|^{-2}\bigr)
\end{equation}
allows us to represent a transition $\theta'\sim\pi_{\gHR}(\theta,\cdot)$ as
\begin{equation}\label{eq:HRnat}
    \C^{-1/2}\theta'\ =\ \Pi_{\C^{-1/2}\rho}\,\C^{-1/2}\theta\ +\ Z\,\C^{-1/2}\rho~~\text{with}~~\rho\ \sim\ \tau\quad\text{and}~~ Z\ \sim\ \mathcal N\bigr(0,\,|\C^{-1/2}\rho|^{-2}\bigr)\;,
\end{equation}
where  $\Pi_w=I_d-|w|^{-2}w\otimes w$, $w\in\mathbb R^d$ is the projection onto the orthogonal complement of $\operatorname{span}(w)$.
We can then construct a coupling $(\theta',\tilde\theta')$ of $\pi_{\gHR}(\theta,\cdot)$ and $\pi_{\gHR}(\tilde\theta,\cdot)$ by relating $\theta'$ to a second copy $\tilde\theta'\sim\pi_{\gHR}(\tilde\theta,\cdot)$ given by
\begin{equation}\label{eq:HRsecond}
    \C^{-1/2}\tilde\theta'\ =\ \Pi_{\C^{-1/2}\tilde\rho}\,\C^{-1/2}\tilde\theta\ +\ \widetilde Z\,\C^{-1/2}\tilde\rho~~\text{with}~~\tilde\rho\ \sim\ \tau~~\text{and}\quad \widetilde Z\ \sim\ \mathcal N\bigr(0,\,|\C^{-1/2}\tilde\rho|^{-2}\bigr)\;.
\end{equation}
In particular, we can use the same auxiliary variables in both transitions, i.e., $\rho=\tilde\rho$ and $Z=\widetilde Z$.
This is referred to as a \emph{synchronous coupling}.
Note that the availability of this coupling hinges on the auxiliary variables' distributions not depending on the initial state.
For this choice, the second terms on the right hand sides of \eqref{eq:HRnat} and \eqref{eq:HRsecond} equal and hence cancel in the difference:
\begin{equation}\label{eq:HRdiff}
    \C^{-1/2}(\theta'-\tilde\theta')\ =\ \Pi_{\C^{-1/2}\rho}\,\C^{-1/2}(\theta-\tilde\theta)\;.
\end{equation}
This shows that, in the coordinates $\C^{-1/2}\theta$ natural to the target $\gamma^\C$, a transition of the difference between the synchronously coupled copies corresponds to a projection onto a random $(d-1)$-dimensional linear subspace.
These random projections contract on average, as quantified by the following general contraction result from \cite[Lemma 1]{BoEbOb2024}.

\begin{lemma}{\cite{BoEbOb2024}}\label{lem:contrproj}
Let $\eta$ be a probability measure on $\mathbb R^d$.
For all $z\in\mathbb R^d$, $\Pi_w$ satisfies
\begin{equation}\label{eq:generalrate}
    \bigr(\mathbb E_{w\sim\eta}|\Pi_w\,z|^2\bigr)^{1/2}\ \leq\ (1-\lambda)\,|z|\quad\text{with}\quad \lambda\ =\ \frac12\inf_{|\zeta|=1}\mathbb E_{w\sim\eta}\Bigr(\zeta\cdot\frac{w}{|w|}\Bigr)^2 \;. 
\end{equation}
\end{lemma}

Taking norms and expectations in \eqref{eq:HRdiff} and subsequently applying Lemma \ref{lem:contrproj} yields a result identical to Theorem \ref{thm:HRcontr} for generalized Hit-and-Run, see \cite{BoEbOb2024}, which is known to be sharp.  Remarkably, the rates for generalized Hit-and-Run and uniform-shifted infinite-orbit NURS coincide.

\begin{figure}
\centering
\begin{tikzpicture}[scale=1.2]
\clip (-0.5,-0.8) rectangle (8.5,4.2);
\draw [line width=1pt,dashed,smooth,samples=100,domain=0:8] plot({\x},{4*e^(-(\x-3.6)^2)});
\draw [dashed] (3.6,0) -- (3.6,4);
\filldraw (3.6,-.1) circle (0pt) node[below] {$-\frac{\theta\cdot\C^{-1}\rho}{|\C^{-1/2}\rho|^2}-s$};
\node[right] at (4.5,{4*e^(-(4.5-3.6)^2)}) {$\gamma^{\C}_{\theta+s\rho,\,\rho}$};
\node[below] at (7,-.1) {$h\mathbb Z$};

\draw (2,0) -- (2,{4*e^(-(2-3.6)^2)});
\filldraw (2,{4*e^(-(2-3.6)^2)}) circle (1.5pt);
\draw (3,0) -- (3,{4*e^(-(3-3.6)^2)});
\filldraw (3,{4*e^(-(3-3.6)^2)}) circle (1.5pt);
\draw (4,0) -- (4,{4*e^(-(4-3.6)^2)});
\filldraw (4,{4*e^(-(4-3.6)^2)}) circle (1.5pt) node[right] {$\mathrm{Law}(T)$};
\draw (5,0) -- (5,{4*e^(-(5-3.6)^2)});
\filldraw (5,{4*e^(-(5-3.6)^2)}) circle (1.5pt);

\draw (0,0) -- (8,0);
\foreach \k in {0,...,8} {\draw[line width=1pt] (\k,-.05) -- (\k,.05);}

\draw[->,line width=2pt] (3,0) -- (3.6,0);
\node[above] at (3.3,0) {$\sigma$};
\end{tikzpicture}
\caption{\textit{Rewriting the scalar displacement $T$ as in \eqref{eq:T}.  Starting from the mean of the scalar displacement measure $\gamma^\C_{\theta+s\rho,\,\rho}$, $-\sigma$ shifts to the lattice $h\mathbb Z$ from where the difference to $T$ is given by the categorical random variable $W$.  The weights of $W$ follow a normal distribution centered at $\sigma$ which quantifies the non-negative deviation of the mean from the lattice $h\mathbb Z$.}}
\label{fig:shift}
\end{figure}
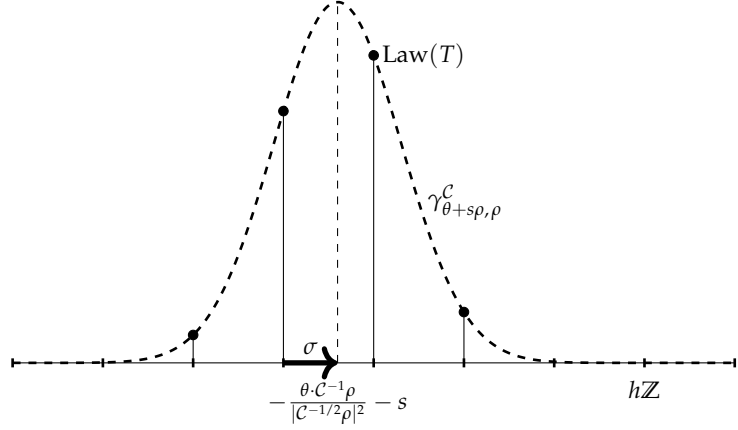

\begin{proof}[Proof of Theorem \ref{thm:HRcontr}]
Based on the above considerations on generalized Hit-and-Run, we now construct a coupling for uniform-shifted infinite-orbit NURS (Algorithm \ref{alg:infty_unadjusted}).
For given $\theta\in\mathbb R^d$, $\rho\in\mathbb S^{d-1}$ and $s\in[-h/2,\,h/2)$, in correspondence to \eqref{eq:TgHR}, the scalar displacement drawn in Run Step 2. can be written as
\begin{equation}\label{eq:T}
    T\ =\ -\frac{\theta\cdot\C^{-1}\rho}{|\C^{-1/2}\rho|^2}\ -\ s\ -\ \sigma\ +\ W
\end{equation}
with
\begin{equation}\label{eq:sigma}
    \sigma\ =\ \inf\Bigr\{t\geq0\ :\ -\frac{\theta\cdot\C^{-1}\rho}{|\C^{-1/2}\rho|^2}-s-t\in h\mathbb Z\Bigr\}
\end{equation}
and
\begin{equation}\label{eq:W}
    W\ \sim\ \cat\Bigr(\mathcal N\bigr(\sigma,\,|\C^{-1/2}\rho|^{-2}\bigr)(t)\Bigr)_{t\in h\mathbb Z}\;.
\end{equation}
This is illustrated in Figure \ref{fig:shift}.
Similarly to \eqref{eq:HRnat}, a transition step $\theta'\sim\pi_{\infty,u}(\theta,\cdot)$ of uniform-shifted infinite-orbit NURS in natural coordinates therefore takes the form
\begin{equation}\label{eq:inftystep}
\C^{-1/2}\theta'\ =\ \Pi_{\C^{-1/2}\rho}\,\C^{-1/2}\theta\ +\ (W-\sigma)\C^{-1/2}\rho
\end{equation}
with $\rho\sim\tau$, $s\sim\mathrm{Unif}([-h/2,\,h/2))$ and $\sigma$, $W$ as defined above.
We again construct a coupling $(\theta',\tilde\theta')$ by relating $\theta'$ to a second copy $\tilde\theta'\sim\pi_{\infty,u}(\tilde\theta,\cdot)$ given by
\begin{equation}\label{eq:infty_second}
\C^{-1/2}\tilde\theta'\ =\ \Pi_{\C^{-1/2}\tilde\rho}\,\C^{-1/2}\tilde\theta\ +\ (\widetilde W-\tilde\sigma)\C^{-1/2}\tilde\rho
\end{equation}
with $\tilde\rho\sim\tau$, $\tilde s\sim\mathrm{Unif}([-h/2,\,h/2))$ and $\tilde\sigma$, $\widetilde W$ defined similar to $\sigma$, $W$ but with $\tilde\theta,\tilde\rho,\tilde s$ instead of $\theta,\rho,s$.
The goal of the construction is, as achieved for generalized Hit-and-Run, to ensure
\begin{equation}\label{eq:goal}
    \C^{-1/2}(\theta'-\tilde\theta')\ =\ \Pi_{\C^{-1/2}\rho}\,\C^{-1/2}(\theta-\tilde\theta)\;.
\end{equation}
For the projections in the two copies to coincide, we again set $\rho=\tilde\rho$.

\begin{figure}
\centering
\begin{tikzpicture}[scale=1]
\clip (-6,-4) rectangle (6,6);

\foreach \k in {3,...,8} {
\filldraw[cyan] ({-7+\k-.1-1.5*e^(-(\k-.1-5.5)^2/2)},{-4+\k-.1+1.5*e^(-(\k-.1-5.5)^2/2)}) circle (1.5pt);
\draw[cyan,dashed] ({-7+\k-.1},{-4+\k-.1}) -- ({-7+\k-.1-1.5*e^(-(\k-.1-5.5)^2/2)},{-4+\k-.1+1.5*e^(-(\k-.1-5.5)^2/2)});
\filldraw[cyan] ({-7+\k+2-.1+1.5*e^(-(\k+2-.1-7.5)^2/2)},{-8+\k+2-.1-1.5*e^(-(\k+2-.1-7.5)^2/2)}) circle (1.5pt);
\draw[cyan,dashed] ({-7+\k+2-.1},{-8+\k+2-.1}) -- ({-7+\k+2-.1+1.5*e^(-(\k+2-.1-7.5)^2/2)},{-8+\k+2-.1-1.5*e^(-(\k+2-.1-7.5)^2/2)});
}
\filldraw[cyan] ({-7+7-.1-1.5*e^(-(7-.1-5.5)^2/2)},{-4+7-.1+1.5*e^(-(7-.1-5.5)^2/2)}) circle (1.5pt) node[above] {$W$};
\filldraw[cyan] ({-7+7+2-.1+1.5*e^(-(7+2-.1-7.5)^2/2)},{-8+7+2-.1-1.5*e^(-(7+2-.1-7.5)^2/2)}) circle (1.5pt) node[right] {$\widetilde W$};

\draw[dashed] (-7,7) -- (7,-7);
\node[below left] at (6,-3) {$\operatorname{span}(\C^{-1/2}\rho)^\perp$};
\draw[line width=1pt] (-7,-4) -- (7,10);
\draw[line width=1pt] (-7,-8) -- (7,6);
\foreach \k in {0,...,14} {
\draw[line width=1pt] ({-7+\k-.05-.1},{-4+\k+.05-.1}) -- ({-7+\k+.05-.1},{-4+\k-.05-.1});
\draw[line width=1pt] ({-7+\k-.05-.1},{-8+\k+.05-.1}) -- ({-7+\k+.05-.1},{-8+\k-.05-.1});
\draw[dashed] ({-7+\k-.1},{-4+\k-.1}) -- ({-7+2+\k-.1},{-8+2+\k-.1});
}

\draw [cyan,dashed,smooth,samples=100,domain=0:14] plot ({-7+\x-1.5*e^(-(\x-5.5)^2/2)},{-4+\x+1.5*e^(-(\x-5.5)^2/2)});
\draw [cyan,dashed,smooth,samples=100,domain=0:14] plot ({-7+\x+1.5*e^(-(\x-7.5)^2/2)},{-8+\x-1.5*e^(-(\x-7.5)^2/2)});

\draw[line width=1.5pt,->] ({-7+4.4-.1},{-4+4.4-.1}) -- ({-7+4-.1},{-4+4-.1}) node[pos=0.5,above left] {$s$};
\draw[line width=1.5pt,->] ({-7+10.7-.1},{-8+10.7-.1}) -- ({-7+11-.1},{-8+11-.1}) node[pos=0.5,above left] {$\tilde s$};
\draw[color=SeaGreen,line width=1.5pt,->] ({-7+5-.1},{-4+5-.1}) -- ({-7+5.5-.1},{-4+5.5-.1}) node[pos=0.5,above left] {$\sigma$};
\draw[color=SeaGreen,line width=1.5pt,->] ({-7+7-.1},{-8+7-.1}) -- ({-7+7.5-.1},{-8+7.5-.1}) node[pos=0.5,above left] {$\tilde\sigma$};
\node[cyan, above left] at ({-7+5-.1},{-4+5-.1}) {$\scriptstyle0$};
\node[cyan, below right] at ({-7+7-.1},{-8+7-.1}) {$\scriptstyle0$};

\filldraw ({-7+4.4-.1},{-4+4.4-.1}) circle (2.5pt) node[right] {$\C^{-1/2}\theta$};
\filldraw ({-7+9-.1},{-8+9-.1}) circle (2.5pt) node[right] {$\C^{-1/2}\tilde\theta$};
\filldraw[color=Blue] ({-7+7-.1},{-4+7-.1}) circle (3pt) node[right] {$\C^{-1/2}\theta'$};
\filldraw ({-7+10.7-.1},{-8+10.7-.1}) circle (2.5pt) node[right] {$\C^{-1/2}\tilde\theta$};
\filldraw[color=Blue] ({-7+9-.1},{-8+9-.1}) circle (3pt) node[right] {$\C^{-1/2}\tilde\theta'$};
\filldraw[color=OliveGreen] ({-7+5.5},{-4+5.5}) circle (2pt) node[right] {$\Pi_{\C^{-1/2}\rho}\C^{-1/2}\theta$};
\filldraw[color=OliveGreen] ({-7+7.5},{-8+7.5}) circle (2pt) node[right] {$\Pi_{\C^{-1/2}\rho}\C^{-1/2}\tilde\theta$};

\node[align=center] at (-2.5,5) {${\color{Blue}\C^{-1/2}\theta'}={\color{OliveGreen}\Pi_{\C^{-1/2}\rho}\,\C^{-1/2}\theta}+({\color{Cyan}W}-{\color{SeaGreen}\sigma})\C^{-1/2}\rho$\\${\color{Blue}\C^{-1/2}\tilde\theta'}={\color{OliveGreen}\Pi_{\C^{-1/2}\rho}\,\C^{-1/2}\tilde\theta}+({\color{Cyan}\widetilde W}-{\color{SeaGreen}\tilde\sigma})\C^{-1/2}\rho$};
\end{tikzpicture}
\caption{\textit{Coupling of two copies $\theta'\sim\pi_{\infty,u}(\theta,\cdot)$ and $\tilde\theta'\sim\pi_{\infty,u}(\tilde\theta,\cdot)$ of uniform-shifted infinite-orbit NURS the transitions of which in natural coordinates as derived in \eqref{eq:inftystep}, \eqref{eq:infty_second} are given in the upper left corner.
From their initial states (black), the copies select their subsequent state (dark blue) from the shifted lattices according to the weights (light blue) determined by the scalar displacement measure centered at the intersection with $\operatorname{span}(\C^{-1/2}\rho)^\perp$.
By coupling the shifts $s,\tilde s$ such that the lattices are aligned relative to $\operatorname{span}(\C^{-1/2}\rho)^\perp$, i.e., $\sigma=\tilde\sigma$, the weights (light blue) for both copies coincide so that their subsequent states (dark blue) can be synchronized.
This allows the coupling to achieve the goal \eqref{eq:goal} of reducing the difference in natural coordinates to the projection onto $\operatorname{span}(\C^{-1/2}\rho)^\perp$.}}
\label{fig:coupling}
\end{figure}
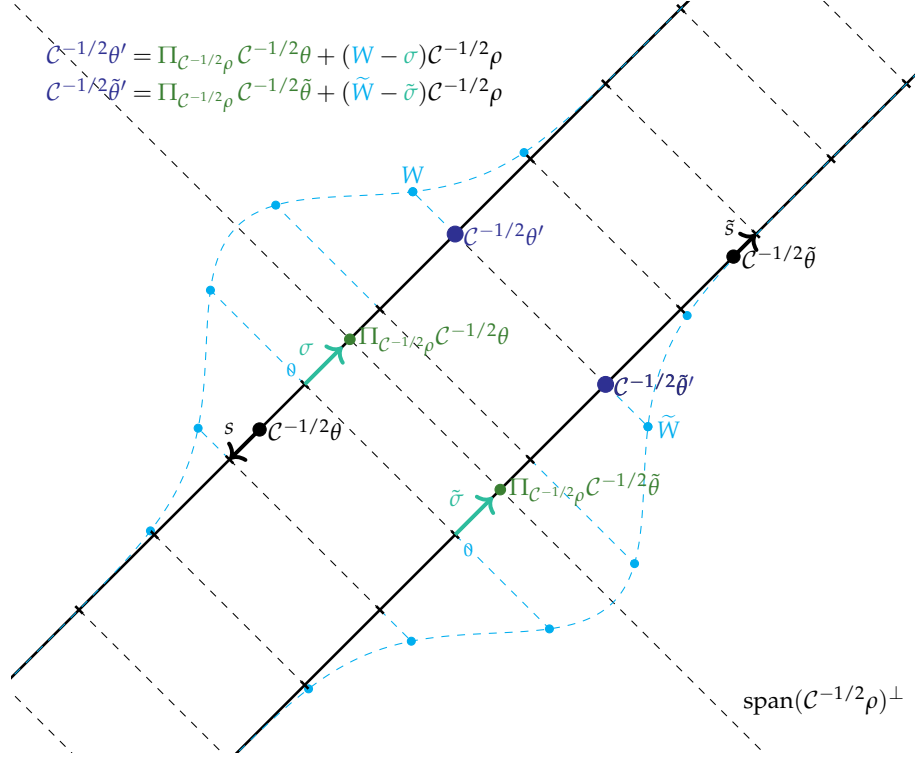

Now, the crucial step in the construction is the coupling of the shifts $s$ and $\tilde s$.
If we were to synchronize them, $\sigma$ and $\tilde\sigma$ would generally not coincide and neither would the distributions of $W$ and $\widetilde W$.
Then, it would be unclear how to synchronize the second terms on the right hand sides of \eqref{eq:inftystep} and \eqref{eq:infty_second}.
Instead, we couple the shifts in such a way that $\sigma=\tilde\sigma$, cf. Figure \ref{fig:coupling}.
As the laws of $W$ and $\widetilde W$ then coincide, we can set $W=\widetilde W$ making the second terms on the right hand sides of \eqref{eq:inftystep} and \eqref{eq:infty_second} equal and therewith achieving \eqref{eq:goal}.

More precisely, given $s\sim\Unif([-h/2,\,h/2))$ we set $\tilde s$ to be the unique element of $[-h/2,\,h/2)$ such that
\[ -\frac{\theta\cdot\C^{-1}\rho}{|\C^{-1/2}\rho|^2}-s\ =\ -\frac{\tilde \theta\cdot\C^{-1}\rho}{|\C^{-1/2}\rho|^2}-\tilde s\quad\text{as elements of $\mathbb R/(h\mathbb Z)$.} \]
In other words, $\tilde s\equiv\frac{(\tilde\theta-\theta)\cdot\C^{-1}\rho}{|\C^{-1/2}\rho|^2}-s \,(\operatorname{mod}h)$ implying $\tilde s\sim\Unif([-h/2,\,h/2))$.
$\sigma=\tilde\sigma$ immediately follows from \eqref{eq:sigma}.

Taking norms in \eqref{eq:goal} yields
\[ |\theta'-\tilde\theta'|_{\C^{-1/2}}\ =\ \bigr|\C^{-1/2}(\theta'-\tilde\theta')\bigr|\ =\ \bigr|\Pi_{\C^{-1/2}\rho}\,\C^{-1/2}(\theta-\tilde \theta)\bigr| \]
so that the general contraction estimate for random projections Lemma \ref{lem:contrproj} finishes the proof:
\begin{align*}
    \W^2_{\C^{-1/2}} & \bigr(\pi_{\infty,u}(\theta, \cdot),\,\pi_{\infty,u}(\tilde \theta, \cdot)\bigr) \\
    \ &\leq\ \Bigr(\mathbb E|\theta'-\tilde \theta'|_{\C^{-1/2}}^2\Bigr)^{1/2}
    \ =\ \Bigr(\mathbb E_{\rho\sim\tau}\bigr|\Pi_{\C^{-1/2}\rho}\,\C^{-1/2}(\theta-\tilde \theta)\bigr|^2\Bigr)^{1/2} \\
    &\leq\ \left(1-\frac12\inf_{|\zeta|=1}\mathbb E_{\rho\sim\tau}\Bigr(\zeta\cdot\frac{\C^{-1/2}\rho}{|\C^{-1/2}\rho|}\Bigr)^2\right)\bigr|\C^{-1/2}(\theta-\tilde \theta)\bigr|\\
    &\leq\ (1-\lambda)|\theta-\tilde \theta|_{\C^{-1/2}}\;.
\end{align*} \end{proof}

\section{Overlap between NURS and Hit-and-Run} \label{sec:hit_and_run}

In this section, we quantify the overlap between uniform-shifted infinite-orbit NURS (Algorithm~\ref{alg:infty_unadjusted}) and generalized Hit-and-Run (Algorithm~\ref{alg:gH&R}).  To do this, we use the total variation distance, a fundamental metric for comparing probability measures, as it captures how much the two measures don't overlap.  

For two probability measures $\eta$ and $\nu$ on $\mathbb{R}^d$, it can be expressed in terms of couplings:
\begin{equation} \label{eq:tv_coupling}
\mathrm{TV}\bigr(\eta, \nu \bigr) \ = \ \inf\,\mathbb P[\theta \neq \tilde{\theta}]
\end{equation}  
where the infimum is taken over all couplings of $\eta$ and $\nu$.  A coupling consists of a pair of random variables $(\theta,\tilde\theta)$ defined on a common probability space such that $\theta\sim\eta$ and $\tilde\theta\sim\nu$.   

Additionally, if $\eta$ and $\nu$ are absolutely continuous  with respect to Lebesgue measure, the total variation distance can also be expressed in terms of their densities: \begin{equation} \label{eq:tv_L1}
\mathrm{TV}\bigr(\eta, \nu \bigr) \ = \ \frac{1}{2} \int_{\mathbb{R}^d} |\eta(x) - \nu(x)| \; dx \;,
\end{equation}
where $\eta(x)$ and $\nu(x)$ are the densities of $\eta$ and $\nu$, respectively.  

The following result provides an upper bound on the total variation distance between the transition kernels of uniform-shifted infinite-orbit NURS and generalized Hit-and-Run, assuming they both start from the same state.  This bound gives a precise measure of how closely the two algorithms update their states at each step.

\begin{theorem} \label{thm:overlap}
Suppose that the target density $\mu$ is differentiable almost everywhere.  For any starting point $\theta \in \mathbb{R}^{d}$ and lattice spacing $h>0$, the total variation distance between the transition kernel of generalized Hit-and-Run $\pi_{\gHR}$ and the transition kernel of uniform-shifted infinite-orbit NURS $\pi_{\infty,u}$ satisfies
\begin{equation}  \label{eq:tv}
\mathrm{TV}\bigr(\pi_{\gHR}(\theta,\cdot),\,  \pi_{\infty,u}(\theta,\cdot)\bigr)\ \le\  2 \, h \,  \int_{\mathbb{R}^d}  | \nabla \log \circ \mu(\theta')| \, \pi_{\gHR}(\theta,d \theta')  \;.
\end{equation}
\end{theorem}

A key ingredient in the proof of this theorem is a piecewise uniform approximation of the scalar displacement measure $\mu_{\theta,\,\rho}$.   To construct this approximation, 
we first define an evenly spaced grid  $\{ t_k = k h \}_{k \in \mathbb{Z}}$ where $h>0$ is the grid spacing. For any $t \in \mathbb{R}$, define the grid-based floor function that maps $t$ to the nearest  grid point less than or equal to $t$, i.e.,  \[
 \lb{t}\ =\ \max \{ k h ~:~ k h \le t \} \;.
\]  
Next, partition $\mathbb{R}$ into subintervals of length $h$, defined as $[t_k^-,t_k^+)$ where
\[
t_k^+-t_k^-\ =\ h \; \quad \text{and} \quad \frac{t_k^+ + t_k^-}{2}\ =\ t_k 
\] for all $k \in \mathbb{Z}$.  On each subinterval, as illustrated in Figure~\ref{fig:pwu}, the piecewise uniform approximation is defined to be constant and proportional to the value of the displacement density at the midpoint of the subinterval.  Specifically, \begin{equation} \label{eq:pwu}
\tilde{\mu}_{\theta,\,\rho}(t)\ =\ \frac{\mu_{\theta,\,\rho}(\lfloor t+h/2 \rfloor_h)}{h Z^{\mathrm{cat}}_{\theta,\,\rho}} 
\end{equation} where we have introduced the normalization constant $Z^{\mathrm{cat}}_{\theta,\,\rho} = \sum_{t \in h \mathbb{Z}} \mu_{\theta,\,\rho}(t)$ for $\cat(h \mathbb{Z}, \mu_{\theta,\,\rho})$.  This normalization constant ensures that $\int_{\mathbb{R}} \tilde{\mu}_{\theta,\,\rho}(t) dt = 1$.  

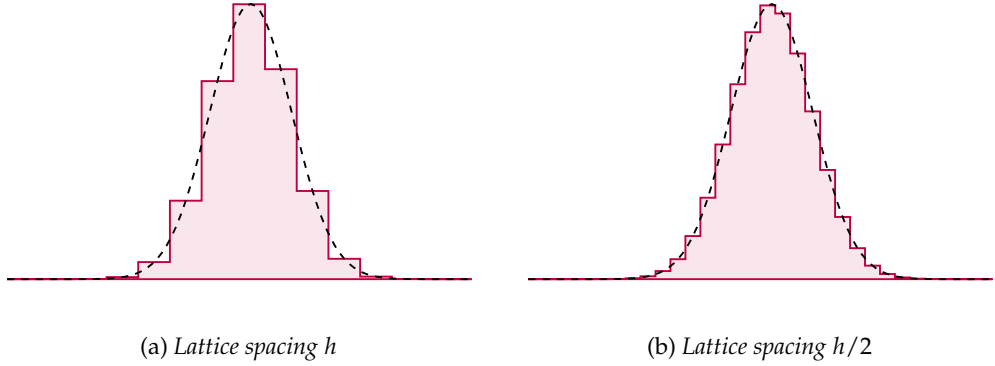
\begin{figure}[t]
\centering
\begin{minipage}{0.48\textwidth}
\centering
\begin{tikzpicture}[scale=0.9]
\begin{axis}[
        xtick=\empty,
        ytick=\empty,
        xticklabel=\empty,
        yticklabel=\empty,
        xlabel={},
        ylabel={},
    xmax=4,ymax=1.2,ymin=-0.2,xmin=-4,
    axis line style={draw=none},
    domain=-4.2:4.2,
    ]

\addplot [draw=purple, fill=purple!10, thick, const plot mark mid, samples=16, domain=-4.2:4]
    {exp(-(x-0.2)^2)}\closedcycle;

\addplot[dashed, thick,domain=-4:4,samples=100]{exp(-(x-0.2)^2)};

\end{axis}
\end{tikzpicture}
\end{minipage}
\begin{minipage}{0.48\textwidth}
\centering
\begin{tikzpicture}[scale=0.9]
\begin{axis}[
        xtick=\empty,
        ytick=\empty,
        xticklabel=\empty,
        yticklabel=\empty,
        xlabel={},
        ylabel={},
    xmax=4,ymax=1.2,ymin=-0.2,xmin=-4,
    axis line style={draw=none},
    domain=-4.2:4.2,
    ]

\addplot [draw=purple, thick, fill=purple!10, const plot mark mid, samples=32, domain=-4:4]
    {exp(-(x-0.2)^2)}\closedcycle;

\addplot[dashed, thick,domain=-4:4,samples=100]{exp(-(x-0.2)^2)};



\end{axis}
\end{tikzpicture}
\end{minipage}
\smallskip
\begin{minipage}{0.48\textwidth}
\centering
{(a) \textit{Lattice spacing $h$}}
\end{minipage}
\begin{minipage}{0.48\textwidth}
\centering
{(b) \textit{Lattice spacing $h/2$}}
\end{minipage}
\caption{\textit{Piecewise uniform approximations of the displacement measure used in Hit-and-Run.}}
\label{fig:pwu}
\end{figure}

\begin{proof}
Given $\rho\sim\tau$, let
\[ T\ \sim\ \mu_{\theta,\,\rho}\quad\text{and}\quad T'\ \sim\ \frac1h\int_{-h/2}^{h/2}\cat\bigr(s+h\mathbb Z, \mu_{\theta,\,\rho}\bigr) \,ds \]
be a maximal coupling.
Then
\[ \theta+T\,\rho\ \sim\ \pi_{\gHR}(\theta,\cdot)\quad\text{and}\quad\theta+T'\,\rho\ \sim\ \pi_{\infty,u}(\theta,\cdot) \]
as $T$ and $T'$ correspond to the scalar displacements by which the respective transitions move along the line away from $\theta$.
Note in particular, that the law of $T'$, in contrast to the distributions inside the integral, is continuously supported.
By the coupling characterization of total variation (see~\eqref{eq:tv_coupling}), maximality of the coupling $(T,T')$, and the triangle inequality,
\begin{align*}
 \mathrm{TV} & \bigr(\pi_{\gHR}(\theta,\cdot),\,  \pi_{\infty,u}(\theta,\cdot)\bigr)
 \  \leq\ \mathbb P(\theta+T\,\rho\ne\theta+T'\,\rho)
 \ \leq\ \mathbb E\,\mathbb P(T\ne T'|\rho) \\
 &\leq\ \mathbb E\,\mathrm{TV}\Bigr(\mu_{\theta,\,\rho},\,\frac1h\int_{-h/2}^{h/2}\cat\bigr(s+h\mathbb Z,\mu_{\theta,\,\rho}\bigr)\,ds\Bigr) \\
 &\leq\ \mathbb E\,\mathrm{TV}(\mu_{\theta,\,\rho},\,\tilde\mu_{\theta,\,\rho})+\mathbb E\,\mathrm{TV}\Bigr(\tilde\mu_{\theta,\,\rho},\,\frac1h\int_{-h/2}^{h/2}\cat\bigr(s+h\mathbb Z,\mu_{\theta,\,\rho}\bigr)\,ds\Bigr)\;.
\end{align*}

The first term on the right hand side satisfies by Lemma~\ref{lem:tv_pwu},
\begin{align*}
    \mathbb E\,\mathrm{TV}(\tilde{\mu}_{\theta,\,\rho},\, \mu_{\theta,\,\rho})\ &\le\  h \; \mathbb E\,\int_{\mathbb{R}} | \nabla \log \circ \mu(\theta + r \rho) |  \;  \mu_{\theta,\,\rho}(r)  \;  dr \\
    \ &\le \ h \,  \int_{\mathbb{R}^d}  | \nabla \log \circ \mu(\theta')| \, \pi_{\gHR}(\theta,d \theta') \;.
\end{align*}

For the second term, let $s\sim\Unif([-h/2,\,h/2))$ and
\[ T\ \sim\ \cat\bigr(h\mathbb Z, \mu_{\theta+s\rho,\,\rho}\bigr)\quad\text{and}\quad \tilde T\ \sim\ \cat\bigr(h\mathbb Z, \mu_{\theta,\,\rho}\bigr) \]
be a maximal coupling.
Then
\[ T+s\ \sim\ \frac1h\int_{-h/2}^{h/2}\cat\bigr(s+h\mathbb Z, \mu_{\theta,\,\rho}\bigr)\,ds\quad\text{and}\quad \tilde T+s\ \sim\ \tilde\mu_{\theta,\,\rho} \]
so that by the coupling characterization of total variation, maximality of the coupling $(T,\tilde T)$, and the fact that the total variation between two discrete probability measures is the $\ell_1$-distance between their probability mass functions (see \eqref{eq:tv_L1}),
\begin{align*}
    &\mathbb E\,\mathrm{TV}\Bigr(\tilde\mu_{\theta,\,\rho},\,\frac1h\int_{-h/2}^{h/2}\cat\bigr(s+h\mathbb Z,\mu_{\theta,\,\rho}\bigr)\,ds\Bigr)
\ \leq\ \mathbb E\,\mathbb P(\tilde T\ne T) \\
&\quad\quad \leq \ \mathbb E\,\mathrm{TV}\Bigr(\cat\bigr(h\mathbb Z, \mu_{\theta,\,\rho}\bigr),\,\cat\bigr(h\mathbb Z, \mu_{\theta+s\rho,\,\rho}\bigr) \Bigr) \\
&\quad\quad\leq\ \frac{1}{2} \mathbb{E}\sum_{t \in h \mathbb{Z}} \left| \frac{\mu_{\theta,\,\rho}(t)}{Z_{\theta,\,\rho}^{\mathrm{cat}}} - \frac{\mu_{\theta+s \rho,\,\rho}(t)}{Z_{\theta+s \rho,\,\rho}^{\mathrm{cat}}} \right|\;.
\end{align*}

By the triangle inequality, \eqref{eq:tv_L1} together with \eqref{eq:pwu}, and Lemma~\ref{lem:tv_pwu}, we further obtain
\begin{align*}
&  \frac{1}{2} \mathbb{E} \left[ \sum_{t \in h \mathbb{Z}} \left| \frac{\mu_{\theta,\,\rho}(t)}{Z_{\theta,\,\rho}^{\mathrm{cat}}} - \frac{\mu_{\theta+s \rho,\,\rho}(t)}{Z_{\theta+s \rho,\,\rho}^{\mathrm{cat}}} \right| \right]   \\
 &\quad \le\  \frac{1}{2} \mathbb{E}  \sum_{t \in h \mathbb{Z}} \left| \frac{\mu_{\theta,\,\rho}(t)}{Z_{\theta,\,\rho}^{\mathrm{cat}}} - \int_{t-h/2}^{t+h/2} \mu_{\theta,\,\rho}(r) dr + \int_{t-h/2}^{t+h/2} \mu_{\theta,\,\rho}(r) dr - \frac{\mu_{\theta+s \rho,\,\rho}(t)}{Z_{\theta+s \rho,\,\rho}^{\mathrm{cat}}}  \right|    \\
 & \quad \le\ \frac{1}{2} \mathbb{E}\sum_{t \in h \mathbb{Z}} \left[ \left| \int_{t-h/2}^{t+h/2} \left( \frac{\mu_{\theta,\,\rho}(t)}{h Z_{\theta,\,\rho}^{\mathrm{cat}}} -  \mu_{\theta,\,\rho}(r)  \right) \; dr \right| + \left| \int_{t-h/2}^{t+h/2} \left( \mu_{\theta,\,\rho}(r)  - \frac{\mu_{\theta+s \rho,\,\rho}(t)}{h Z_{\theta+s \rho,\,\rho}^{\mathrm{cat}}} \right) \; dr \right|  \right]  \\ 
  & \quad \le\ \frac{1}{2} \mathbb{E}\sum_{t \in h \mathbb{Z}} \left[ \int_{t-h/2}^{t+h/2} \left|  \frac{\mu_{\theta,\,\rho}(t)}{h Z_{\theta,\,\rho}^{\mathrm{cat}}} -  \mu_{\theta,\,\rho}(r) \right| \; dr  +  \int_{t-h/2}^{t+h/2} \left| \mu_{\theta,\,\rho}(r)  - \frac{\mu_{\theta+s \rho,\,\rho}(t)}{h Z_{\theta+s \rho,\,\rho}^{\mathrm{cat}}} \right| \; dr   \right]  \\ 
    & \quad \le\ \frac{1}{2} \mathbb{E} \left[  \int_{-\infty}^{\infty} \left|  \frac{\mu_{\theta,\,\rho}(\lfloor t+h/2 \rfloor_h)}{h Z_{\theta,\,\rho}^{\mathrm{cat}}} -  \mu_{\theta,\,\rho}(r) \right| \; dr  +  \int_{-\infty}^{\infty}  \left| \mu_{\theta,\,\rho}(r)  - \frac{\mu_{\theta+s \rho,\,\rho}(\lfloor t+h/2 \rfloor_h)}{h Z_{\theta+s \rho,\,\rho}^{\mathrm{cat}}} \right| \; dr   \right]  \\ 
    & \quad \le\ \frac{1}{2} \mathbb{E} \Bigr[ \mathrm{TV}(\tilde{\mu}_{\theta,\,\rho},\, \mu_{\theta,\,\rho}) + \mathrm{TV}(\tilde{\mu}_{\theta+s \rho,\,\rho},\, \mu_{\theta,\,\rho}) \Bigr] \\
    &\quad\le\ h \; \mathbb E\,\int_{\mathbb{R}} | \nabla \log \circ \mu(\theta + r \rho) |  \;  \mu_{\theta,\,\rho}(r)  \;  dr \;.
\end{align*}
This completes the proof.
\end{proof}

The following lemma provides an explicit bound on the total variation distance between the one-dimensional displacement density $\mu_{\theta,\,\rho}$ and its piecewise uniform approximation $\tilde{\mu}_{\theta,\,\rho}$ from \eqref{eq:pwu} shifted by $s \in [-h/2,h/2]$.

\begin{lemma}
\label{lem:tv_pwu}
Suppose that the target density $\mu$ is differentiable almost everywhere.   Then for any $s \in [-h/2,h/2]$, $\theta \in \mathbb{R}^d$, $\rho \in \mathbb{S}^{d-1}$, and $h>0$, \[
 \mathrm{TV}(\tilde{\mu}_{\theta+s \rho,\,\rho},\, \mu_{\theta,\,\rho})\ \le\  h \; \int_{\mathbb{R}} | \nabla \log \circ \mu(\theta + r \rho) |  \;  \mu_{\theta,\,\rho}(r)  \;  dr \;.
\] 
\end{lemma}

\begin{proof}
 Using the triangle inequality, and the definition of the piecewise uniform approximation in \eqref{eq:pwu}, \begin{align}
   &  \mathrm{TV}(\tilde{\mu}_{\theta+s \rho,\,\rho}, \mu_{\theta,\,\rho})  
  \ \overset{\eqref{eq:tv_L1}}{=}\ \frac{1}{2} \int_{-\infty}^{\infty} \left| \tilde{\mu}_{\theta+s \rho,\, \rho}(t)  - \mu_{\theta,\,\rho}(t) \right| dt   \nonumber \\
   & \quad =\ \frac{1}{2}  \int_{-\infty}^{\infty} \left| \frac{\mu_{\theta+s \rho,\, \rho}(\lfloor t+h/2 \rfloor_h)}{h Z^{\mathrm{cat}}_{\theta+s \rho,\, \rho}} - \mu_{\theta+s \rho,\, \rho}(\lfloor t+h/2 \rfloor_h) \right. \nonumber  \\
   & \qquad\qquad\qquad \qquad \left. \vphantom{\frac{\mu_{\theta+s \rho,\, \rho}(\lfloor t+h/2 \rfloor_h)}{h Z^{\mathrm{cat}}_{\theta+s \rho,\, \rho}} } + \mu_{\theta+s \rho, \,\rho}(\lfloor t+h/2 \rfloor_h) - \mu_{\theta,\,\rho}(t) \right| dt   \nonumber \\
   &\quad \le\ \frac{1}{2}  \left| h Z^{\mathrm{cat}}_{\theta+s \rho,\, \rho} - 1  \right|  +  \frac{1}{2} \int_{-\infty}^{\infty}  \left| \mu_{\theta+s \rho, \,\rho}(\lfloor t+h/2 \rfloor_h) - \mu_{\theta,\,\rho}(t) \right| \, dt    \nonumber \\ 
    &\quad  \le\  \int_{-\infty}^{\infty}  \left| \mu_{\theta+s \rho,\, \rho}(\lfloor t+h/2 \rfloor_h) - \mu_{\theta,\,\rho}(t) \right| \, dt \;. \label{ieq:gentv}
    \end{align}

To estimate \eqref{ieq:gentv}, we use the fundamental theorem of calculus, \begin{equation}
\label{eq:expan}
\begin{aligned}
\mu_{\theta,\,\rho}(t) - \mu_{\theta+s \rho,\, \rho}(\lfloor t+h/2 \rfloor_h) \ &=\  \mu_{\theta,\,\rho}(t) - \mu_{\theta,\,\rho}(s+\lfloor t+h/2 \rfloor_h) \nonumber  \\
&=\ \int_{s + \lfloor t + h/2 \rfloor_h}^t \nabla \log \circ \mu (\theta + r \rho) \cdot \rho \; \mu_{\theta,\,\rho}(r) \; dr  \;.
\end{aligned}
\end{equation}  Hence, \eqref{ieq:gentv} can be expressed as \begin{align*}
\mathrm{TV}(\tilde{\mu}_{\theta+s \rho,\,\rho},\, \mu_{\theta,\,\rho})\  &\le\  \int_{-\infty}^{\infty}  \left|  \int_{ s + \lfloor t + h/2 \rfloor_h}^t \nabla \log \circ \mu (\theta + r \rho) \cdot \rho \; \mu_{\theta,\,\rho}(r) \; dr  \right| dt  \\
 &\le\  \int_{-\infty}^{\infty}  \int_{s+ \lfloor t + h/2 \rfloor_h}^t  \left| \nabla \log \circ \mu (\theta + r \rho) \right| \; \mu_{\theta,\,\rho}(r)   \; dr \; dt  \;, 
\end{align*}
where in the last step we used the triangle inequality and $|\rho|=1$.  
Moreover, since $s+ t_k \in [t_k^-, t_k^+)$,   \begin{align}
 \mathrm{TV}(\tilde{\mu}_{\theta+s \rho,\,\rho},\, \mu_{\theta,\,\rho})  \ & \le\  \sum_{k=-\infty}^{\infty} \int_{t_k^-}^{t_k^+}  \int_{s+t_k}^t  \left| \nabla \log \circ \mu (\theta + r \rho) \right| \; \mu_{\theta,\,\rho}(r)   \; dr \; dt  \nonumber \\
& \le\   \sum_{k=-\infty}^{\infty}   \int_{t_k^-}^{t_k^+}  \int_{t_k^-}^{t_k^+} \left| \nabla \log \circ \mu (\theta + r \rho) \right| \; \mu_{\theta,\,\rho}(r)   \; dr \; dt  \nonumber \\
&  \le\   \; h \; \sum_{k=-\infty}^{\infty}   \int_{t_k^-}^{t_k^+} \left| \nabla \log \circ \mu (\theta + r \rho) \right| \; \mu_{\theta,\,\rho}(r)   \; dr\nonumber\\
&=\   \; h \;   \int_{-\infty}^{\infty} \left| \nabla \log \circ \mu (\theta + r \rho) \right| \; \mu_{\theta,\,\rho}(r)   \; dr\;.  \nonumber 
\end{align}
This completes the proof.
\end{proof}

\section{Analysis of NURS in Neal's Funnel}\label{sec:funnel}

In this section, we examine the performance of NURS in the context of Neal's funnel -- a canonical example that highlights the challenges commonly encountered with Bayesian hierarchical models \cite{Neal2003Slice, betancourt2013hamiltonianmontecarlohierarchical}.  We begin by describing Neal's funnel and its key features. We then detail the tuning of NURS.  Next, we  discuss mixing results for RWM and Hit-and-Run reported in the literature.  Using NURS' connection to both methods, we derive theoretical insights into the behavior of NURS and RWM.  Finally, we evaluate NURS empirically in the funnel.

\subsection{Neal's Funnel-shaped Distribution}

Neal's funnel is a distribution defined on $\mathbb R^{d+1}$ with density given by
\begin{equation} \label{eq:original-funnel}
p(\omega, x_1, \dots, x_d)\ =\ \text{normal}(\omega \mid 0,\,9) \prod_{i=1}^d \text{normal}(x_i \mid 0,\,e^{\omega}) \;,
\end{equation}
where $\text{normal}(\,\cdot \mid 0,\,\sigma^2)$ denotes the density of a centered one-dimensional normal with variance $\sigma^2$. 

The geometry of Neal's funnel consists of two distinct regions (see Figure~\ref{fig:funnel}): 
\begin{itemize}
\item The neck: a high-curvature region, with all but one direction tightly concentrated. 
\item The mouth: a broad, flat region in $d$ directions, with one direction at unit scale.  
\end{itemize} This combination creates a globally ill-conditioned target, with large differences in scales between the neck and mouth regions.  These multiple scales present significant challenges for efficient sampling.

To gain intuition into the behavior of NURS and RWM in this geometry, we locally approximate the neck and mouth regions using $(d+1)$-dimensional Gaussian distributions: $\mathcal N(0,\C_{\mathrm{neck}})$ and $\mathcal N(0,\C_{\mathrm{mouth}})$ with covariance matrices defined as:
\begin{equation}\label{eq:pwGaussian}
    \C_{\mathrm{neck}}\ =\ \operatorname{diag}(1,L^{-1},\dots)~~\text{and}\quad\C_{\mathrm{mouth}}\ =\ \operatorname{diag}(1,m^{-1},\dots)~~\text{for}~~ 0<m\ll L\;.
\end{equation}
By studying the behavior of NURS and RWM in these  Gaussian distributions, we can gain theoretical insight into their performance in the more complex geometry of Neal's funnel.

\begin{figure}
\centering
\begin{tikzpicture}[scale=1]
\foreach \sgn in {-,+}
    \draw[line width=1pt] plot[samples=100,domain=-3:3] ({\sgn e^(-\x/2)},\x);
\draw[line width=1pt] ({-e^(3/2)},-3) arc (130:270:0.15) coordinate[pos=1] (end1) {};
\draw[line width=1pt] ({e^(3/2)},-3) arc (40:-90:0.15) coordinate[pos=1] (end2) {};
\draw[line width=1pt] (end1) -- (end2);
\draw[line width=1pt] ({-e^(-3/2)},3) arc (170:10:0.226);
\draw[dashed] (0,-2.6) ellipse (3.3 and 0.5);
\draw[dashed] (0,2.1) ellipse (0.22 and 1);

\node[left] at (-0.5,2.55) {\large Neck};
\node[left] at (-3.5,-2) {\large Mouth};

\node[right] at (0.35,2.5) {$\mathcal N(0,\C_{\mathrm{neck}})$};
\node[right] at (3,-2) {$\mathcal N(0,\C_{\mathrm{mouth}})$};

\draw[->] (-4,0) -- ({1.5-4},0) node[pos=1,below right] {$\mathbb R^d$};
\draw[->] (-4,0) -- (-4,1.5) node[pos=1,above left] {$\mathbb R$};

\draw[|-|] (-0.22,3.5) -- (0.22,3.5) node[pos=0.5,above] {$L^{-1/2}$};
\draw[|-|] (-3.3,-3.7) -- (3.3,-3.7) node[pos=0.5,below] {$m^{-1/2}$};

\end{tikzpicture}
\caption{\textit{Locally approximating neck and mouth of Neal's funnel by Gaussian distributions.}}
\label{fig:funnel}
\end{figure}
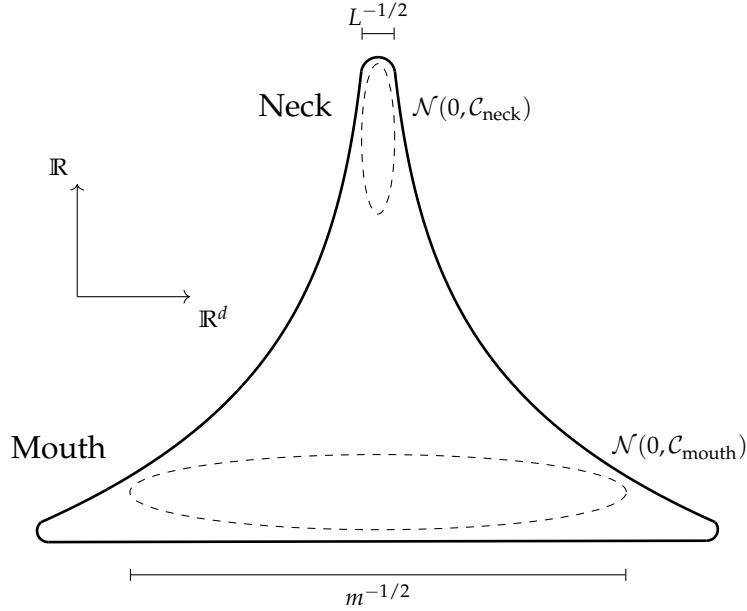

\subsection{Tuning of NURS and RWM}

Sampling methods can, in principle, be made invariant to local changes in geometry by locally adapting their transitions to the structure of the target distribution. However, achieving this in practice requires detailed knowledge about the target distribution (e.g., curvature information) that is typically computationally expensive to obtain.  While NURS, like RWM, can theoretically incorporate local preconditioning (e.g., by locally adapting the direction distribution $\tau$ to reflect the local geometry), we leave such extensions to future work.   In the following analysis, we focus on NURS using a uniform direction distribution, i.e, $\tau=\Unif(\mathbb S^{d-1})$.  Nevertheless, NURS  incorporates a form of local adaptation through its No-Underrun-based orbit selection. As we will show, this local adaptation reduces sampling complexity while relying only on information readily available during sampling.

For the sake of comparison, we assume that both NURS and RWM are tuned prior to sampling using information obtained during a pre-processing or warm-up phase. Specifically, the  lattice spacing for  NURS and step size for RWM are tuned to ensure that the Metropolis acceptance probabilities are lower-bounded in the bulk of the target distribution. Following the tuning guidelines derived in Section \ref{sec:RWM}, and considering the local Gaussian approximations to Neal's funnel described in \eqref{eq:pwGaussian}, the required lattice spacing for NURS and step size for RWM both scale as $h=O(L^{-1/2})$.   Since the precise value of $L$ might not be known in practice, we set $h=\varsigma L^{-1/2}$, where $\varsigma\leq1$ acts as a safety factor  to reduce the risk of exponential slowing caused by an excessively large $h$.

\subsection{Mixing of RWM}  Using conductance methods, Ref.~\cite{andrieuAAP} derive mixing time bounds for RWM.
This approach combines two components: (i) controlling the Metropolis step to reduce the dynamics to the underlying random walk and (ii) estimating bottlenecks along the random walk's path.
As a result, RWM, with the step size $h=\varsigma L^{-1/2}$, achieves a mixing rate $\lambda$ of order
\begin{equation}\label{eq:RWMrate}
\lambda\ \asymp\ \varsigma^2(L d)^{-1}\quad\text{in}\quad\mathcal N(0,\C_{\mathrm{neck}})\;,\quad\text{and}\quad \lambda\ \asymp\ \varsigma^2(\kappa d)^{-1}\quad\text{in}\quad\mathcal N(0,\C_{\mathrm{mouth}})
\end{equation}
where $\kappa=L/m$.  RWM makes small, local steps uniformly throughout the funnel, resulting in diffusive mixing.  In the larger mouth region, mixing is slower than in the neck (see Figure~\ref{fig:funnel_steps} (b)).

\subsection{Mixing of Hit-and-Run}

Via coupling methods, Ref.~\cite{BoEbOb2024} analyze the mixing properties of Hit-and-Run in Gaussian distributions.\footnote{Ref.~\cite{ascolani2024} shows relative entropy contraction for Hit-and-Run obtaining the rate in the neck but not the speed-up in the mouth.}
The results show that Hit-and-Run mixes with rates
\begin{equation}\label{eq:HRrate_general}
    \lambda\ \asymp\ (Ld)^{-1}~~\text{in}~~\mathcal N(0,\C_{\mathrm{neck}})\;,~~\text{and}~~(m^{-1/2}d)^{-1}\ \lesssim\ \lambda\ \lesssim\ d^{-1}\quad\text{in}~~\mathcal N(0,\C_{\mathrm{mouth}})\;,
\end{equation}
where $x\lesssim y$ iff $x\leq Cy$ for some $C>0$.  In the neck, similarly to RWM, Hit-and-Run mixes diffusively (see Figure~\ref{fig:funnel_steps} (a)).  In the mouth, the smaller rate is realized in low dimension, $d=2,3$, and describes ballistic mixing, an ability of Hit-and-Run uncovered in \cite{BoEbOb2024}.  For higher dimension, the convergence rate improves further towards the rate in isotropic distributions.

Unlike RWM, Hit-and-Run mixes faster in the mouth region than in the neck region. This improvement is due to Hit-and-Run's ability to make larger moves in the flat mouth region, whereas RWM is limited to the same local moves it employs in the neck region (see Figure \ref{fig:funnel_steps}).
Since NURS operates in the Hit-and-Run regime in the mouth region (see Figure \ref{fig:regimes} (a)), it inherits this advantage, leading to a quantitative speed-up compared to RWM in this region.

\begin{figure}
\centering
\begin{minipage}{0.48\textwidth}
\centering
\begin{tikzpicture}[scale=.72]
\clip (-5.5,-4) rectangle (5.5,3.5);
\foreach \sgn in {-,+}
    \draw[line width=1pt] plot[samples=100,domain=-3:3] ({\sgn e^(-\x/2)},\x);
\draw[line width=1pt] ({-e^(3/2)},-3) arc (130:270:0.15) coordinate[pos=1] (end1) {};
\draw[line width=1pt] ({e^(3/2)},-3) arc (40:-90:0.15) coordinate[pos=1] (end2) {};
\draw[line width=1pt] (end1) -- (end2);
\draw[line width=1pt] ({-e^(-3/2)},3) arc (170:10:0.226);

\draw[gray,shift={(0,-5.5)}] (0.75,8.6) -- (-0.75,8.1);
\draw[gray,shift={(0,-5.5)}] (-0.75,8.5) -- (0.75,8.0);
\draw[gray,shift={(0,-5.5)}] (0.45,8.3) -- (-0.95,7.85);

\draw[shift={(0,-5.5)}] (0.15,8.4) -- (-0.15,8.3) -- (0.15,8.2) -- (-0.15,8.1);
\filldraw[shift={(0,-5.5)}] (0.15,8.4) circle (1pt);
\filldraw[shift={(0,-5.5)}] (-0.15,8.3) circle (1pt);
\filldraw[shift={(0,-5.5)}] (0.15,8.2) circle (1pt);
\filldraw[shift={(0,-5.5)}] (-0.15,8.1) circle (1pt);

\draw[gray] plot[samples=100,domain=-3:4] ({-0.15+\x},{-2.9+\x/3});
\draw[gray] plot[samples=100,domain=-5:3] ({-0.15+2+\x},{-2.9+2/3-\x/3});
\draw[gray] plot[samples=100,domain=-5:3] ({-0.15+2-2+\x},{-2.9+2/3+2/3+\x/3});

\filldraw (-0.15,-2.9) circle (1pt);
\filldraw ({-0.15+2},{-2.9+2/3}) circle (1pt);
\filldraw ({-0.15+2-2},{-2.9+2/3+2/3}) circle (1pt);
\filldraw ({-0.15+2-2-2.5},{-2.9+2/3+2/3-2.5/3}) circle (1pt);

\end{tikzpicture}
\caption*{(a) \textit{NURS}}
\end{minipage}
\begin{minipage}{0.48\textwidth}
\centering
\begin{tikzpicture}[scale=0.72]
\clip (-5.5,-4) rectangle (5.5,3.5);
\foreach \sgn in {-,+}
    \draw[line width=1pt] plot[samples=100,domain=-3:3] ({\sgn e^(-\x/2)},\x);
\draw[line width=1pt] ({-e^(3/2)},-3) arc (130:270:0.15) coordinate[pos=1] (end1) {};
\draw[line width=1pt] ({e^(3/2)},-3) arc (40:-90:0.15) coordinate[pos=1] (end2) {};
\draw[line width=1pt] (end1) -- (end2);
\draw[line width=1pt] ({-e^(-3/2)},3) arc (170:10:0.226);

\draw[shift={(0,-11)}] (0.15,8.4) -- (-0.15,8.3) -- (0.15,8.2) -- (-0.15,8.1);
\filldraw[shift={(0,-11)}] (0.15,8.4) circle (1pt);
\filldraw[shift={(0,-11)}] (-0.15,8.3) circle (1pt);
\filldraw[shift={(0,-11)}] (0.15,8.2) circle (1pt);
\filldraw[shift={(0,-11)}] (-0.15,8.1) circle (1pt);

\draw[shift={(0,-5.5)}] (0.15,8.4) -- (-0.15,8.3) -- (0.15,8.2) -- (-0.15,8.1);
\filldraw[shift={(0,-5.5)}] (0.15,8.4) circle (1pt);
\filldraw[shift={(0,-5.5)}] (-0.15,8.3) circle (1pt);
\filldraw[shift={(0,-5.5)}] (0.15,8.2) circle (1pt);
\filldraw[shift={(0,-5.5)}] (-0.15,8.1) circle (1pt);

\end{tikzpicture}
\caption*{(b) \textit{RWM}}
\end{minipage}
\caption{\textit{Three transitions of NURS and RWM in both the neck and mouth regions of Neal's funnel.  NURS leverages its orbits (gray), which are locally adapted to the scale of the funnel, enabling it to achieve larger moves in the mouth region.  This yields a considerable speed-up of NURS over RWM both in terms of total complexity, as well as number of sequential computations.
}}
\label{fig:funnel_steps}
\end{figure}

\subsection{NURS vs. RWM in the funnel}

We now analyze the computational complexities for NURS and RWM in the two regions of the funnel.  A reasonable measure of complexity is the asymptotic number of evaluations of the potential $U=-\log\circ\mu$ required to achieve constant sampling accuracy.
For both methods the total complexity is the product of the number of transitions and the complexity per transition.
The number of transitions is given by the inverse of the mixing rate.
The complexity per transition is $1$ potential evaluation for RWM, while for NURS, it depends on the number of states in the orbit selected according to the No-Underrun condition.
The orbit length scales with the local scale of the bulk which is $L^{-1/2}$ in the neck and $m^{-1/2}$ in the mouth (see Figure~\ref{fig:funnel}). Dividing the orbit length  by the lattice spacing $h=\varsigma L^{-1/2}$, the number of states is then of order $\varsigma^{-1}$ in the neck and $\varsigma^{-1}\kappa^{1/2}$ in the mouth.
This reflects a larger per-transition complexity of NURS in the mouth compared to the neck, due to longer orbits adapted to the local scale of the target (see Figure \ref{fig:funnel_steps}).
The resulting projected complexities for NURS and RWM in both regions are summarized in Table \ref{tab:funnel}.

\begin{table}[ht]
\centering
\begin{tabular}{||c||ccccc||ccccc||}
\cline{2-11}
\multicolumn{1}{c||}{} & \multicolumn{5}{|c||}{\textbf{NURS}} & \multicolumn{5}{|c||}{\textbf{RWM}} \\
\multicolumn{1}{c||}{} & total & $=$ & non-par. & $\times$ & parallel & total & $=$ & non-par. & $\times$ & parallel \\
\hline
\textbf{Neck} & $\varsigma^{-1}\,L\phantom{^{-1/2}\,\kappa^{1/2}}\,d$ & $=$ & $L\phantom{^{-1/2}}\,d$ & $\times$ & $\varsigma^{-1}\,\phantom{\kappa^{1/2}}$ & $\varsigma^{-2}\,L\,d$ & $=$ & $\varsigma^{-2}\,L\,d$ & $\times$ & $1$ \\
\hline
\textbf{Mouth} & $\varsigma^{-1}\,m^{-1/2}\,\kappa^{1/2}\,d$ & $=$ & $m^{-1/2}\,d$ & $\times$ & $\varsigma^{-1}\,\kappa^{1/2}$ & $\varsigma^{-2}\,\kappa\,d$ & $=$ & $\varsigma^{-2}\,\kappa\,d$ & $\times$ & $1$ \\
\hline
\end{tabular}
\caption{ \textit{Projected complexities  of NURS vs.~RWM, measured in terms of the number of potential evaluations needed to achieve a predefined accuracy in the mixing time, with $h=\varsigma L^{-1/2}$ based on the local Gaussian approximations \eqref{eq:pwGaussian}.  The total complexities are divided into their parallelizable and non-parallelizable fractions.  The displayed complexities for NURS in the mouth correspond to the low dimensional case and further improve by up to a factor of $m^{1/2}$ in high dimension.}}
\label{tab:funnel}
\end{table}

From Table \ref{tab:funnel}, we observe that NURS achieves a parallelizable ballistic complexity in the inverse safety factor $\varsigma^{-1}$, whereas RWM remains non-parallelizable and exhibits diffusive scaling in this parameter. (Here ``ballistic'' refers to complexity proportional to $\varsigma^{-1}$, while ``diffusive'' refers to a complexity proportional to $(\varsigma^{-1})^2$.)  For well-tuned $h$, the following performance patterns are anticipated:
\begin{enumerate}
\item In the high-curvature neck region, NURS and RWM share comparable performance.
\item In the flat mouth region, NURS considerably outperforms RWM both in terms of total complexity and the number of sequential computations required.
\end{enumerate}
These results suggest that NURS has the potential to outperform RWM in multi-scale target distributions like Neal's funnel.

\subsection{Empirical Evaluation}

We evaluate NURS numerically in Neal's funnel-shaped distribution \eqref{eq:original-funnel} in moderate dimension $d=10$.
In the simulations, we fix the threshold $\epsilon=0$, effectively disabling the No-Underrun condition. The $2^{M}$ log densities within each orbit are evaluated in parallel in a single batch.

\paragraph*{Choice of $h$ and $M$}
The lattice spacing tuning guidance developed in Section~\ref{sec:RWM} suggests $h=O(L^{-1/2})$ where $L^{-1/2}$ represents the most restrictive scale of the target.
In the funnel, the most restrictive scales occur orthogonal to the funnel axis within the neck, becoming increasingly concentrated for negative values of $\omega$.
To estimate the most restrictive scale relevant to simulation, i.e., within the bulk of the funnel, we consider the first $0.001$-quantile of the funnel axis marginal, which lies approximately at $\omega=-9$.
The scale, being comparable to the standard deviations $e^{\omega/2}$ in the remaining directions, is of the order $0.01$ at this $\omega$.  The tuning guidelines therefore suggest choosing $h$ on the order of $0.01$.
As shown in Figure~\ref{fig:nur_stats_omega} (b), this choice is reasonable, as the Metropolis acceptance rate, for larger $h$, starts to degenerate in the neck.

We pick $M=14$ leading to orbits of length $2^M h$ due to the disabled No-Underrun condition.  As shown in Figure~\ref{fig:nur_stats_omega} (a), the jump distances level off in the mouth which shows that the orbit length is too short for the algorithm to realize its full potential of making larges moves adapted to the target distribution's local scale.  This illustrates the importance of local orbit-length adaptation via the No-Underrun condition: The fixed-length orbits waste computational resources in the neck while slowing space exploration in the mouth.

Figure~\ref{fig:nurs_scatter} presents a scatter plot of the samples obtained from running $10^7$ iterations, while Figure~\ref{fig:nur_stats_omega} (c) shows the corresponding histogram of the funnel axis variable $\omega$. The results indicate that the samples closely align with the funnel distribution. Further, implementing the No-Underrun-based orbit-length adaptation will considerably increase efficiency.

\begin{figure}
\centering
\begin{minipage}{0.32\textwidth}
\centering
\includegraphics[scale=0.38]{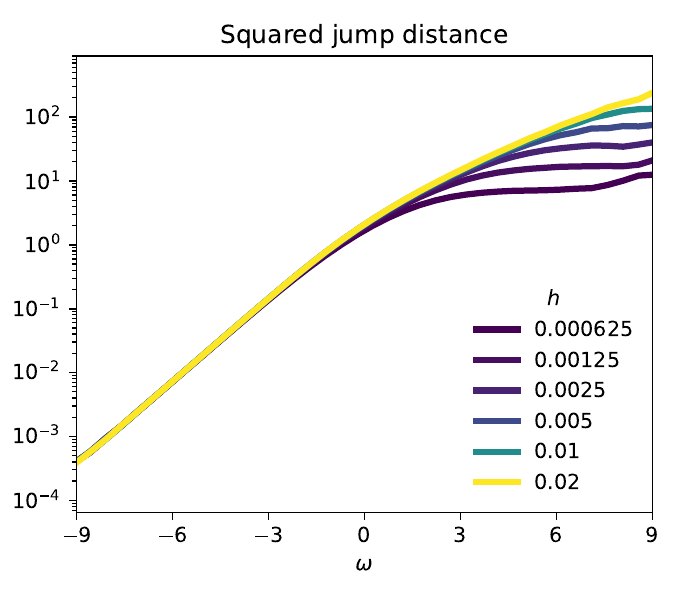}
\vspace{-15pt}
\caption*{(a)}
\end{minipage}
\begin{minipage}{0.32\textwidth}
\centering
\includegraphics[scale=0.38]{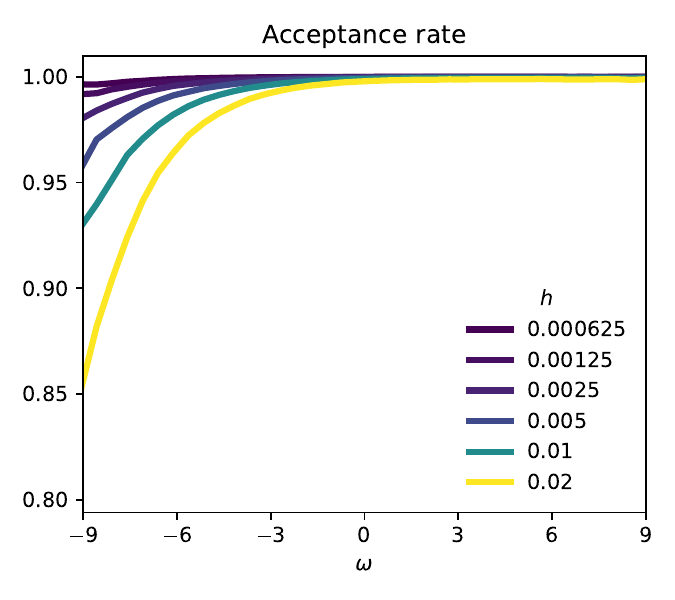}
\vspace{-15pt}
\caption*{(b)}
\end{minipage}
\begin{minipage}{0.32\textwidth}
\centering
\vspace{15pt}
\includegraphics[scale=0.48]{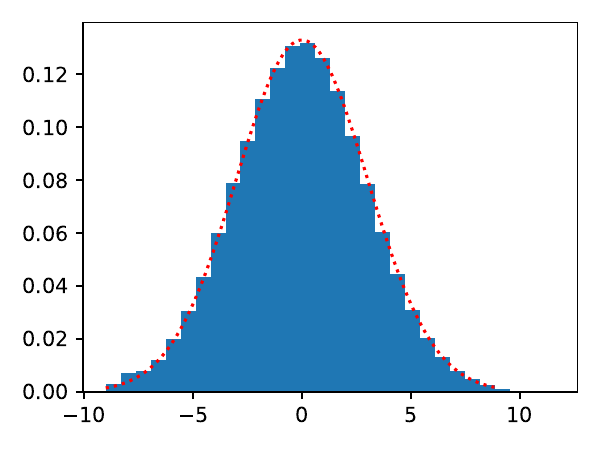}
\vspace{-25pt}
\caption*{(c)}
\end{minipage}
\caption{\textit{\textrm{(a)} Squared jump distances (on a log scale) between consecutive states versus $\omega$; \textrm{(b)} Metropolis acceptance rates versus $\omega$; \textrm{(c)} Histogram of the funnel axis variable $\omega$.}}
\label{fig:nur_stats_omega}
\end{figure}

\begin{figure}
\centering
\includegraphics[scale=.8]{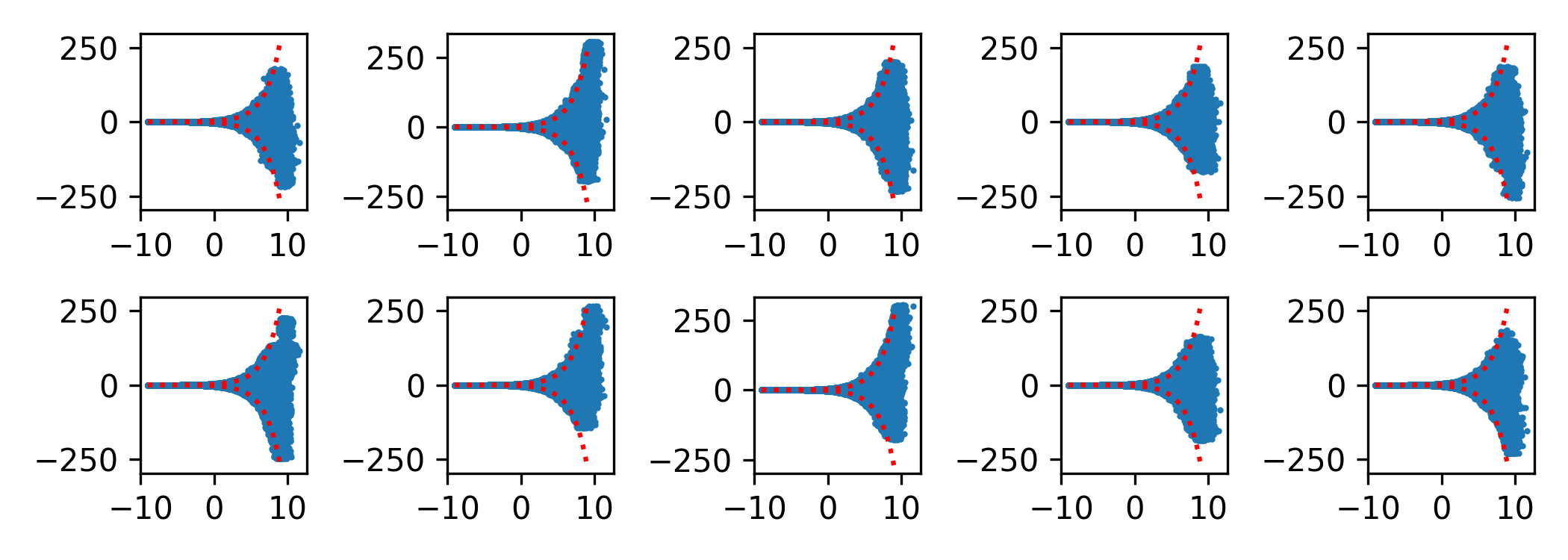}  
\caption{\textit{Scatter plots in the plane defined by the funnel axis variable $\omega$ on the horizontal axis and each of the remaining ten variables on the vertical axis. The red dotted lines represent $\pm3 e^{\omega/2} $.}}
\label{fig:nurs_scatter}
\end{figure}

\section{Proof of Reversibility} \label{sec:reversibility}

Here we present a self-contained proof of the reversibility of NURS (Algorithm~\ref{alg:nurs}) with respect to the target distribution $\mu$.  It follows the ideas developed in \cite{BouRabeeCarpenterMarsden2024} interpreting the transition of NURS as the marginal of a transition on an enlarged space comprising all auxiliary variables used in a NURS transition, such as direction $\rho$ and orbit $\mathcal O$.  We first establish that the transition on this extended space is reversible with respect to a natural distribution whose marginal is $\mu$.  This, in turn, implies that the marginal transition corresponding to NURS is itself reversible with respect to the marginal distribution $\mu$.

\begin{proof}[Proof of Theorem~\ref{thm:NURS_reversibility}]
Since the initial Metropolis-adjusted shift in Algorithm~\ref{alg:nurs} is reversible by design, we focus on the remaining steps.  To analyze their reversibility, we first identify the auxiliary variables involved:
NURS samples a direction $\rho\sim \tau$ and an orbit $\calO\sim p_{\mathrm{Orbit}}(\,\cdot \mid \theta,\rho)$. 
Since the orbit is constructed via the doubling procedure, the orbit size is a power of 2: $|\calO|=2^\ell$ where $\ell\in\mathbb N$. Denote the right endpoint of $\calO$ by $\theta+\rho h R$ for $R\in\mathbb N$. Given $(\theta,\rho)$, the orbit $\calO$ is uniquely determined by its size and its right endpoint. Therefore, we use $\calO$ and $(\ell,R)$ interchangeably, and write $(\ell,R)\sim p_{\orbit}(\,\cdot\mid\theta,\rho)$.
Given an orbit specified by $(\theta,\rho,\ell,R)$, an index $i$ is sampled, whose density with respect to counting measure on $ \mathbb Z$ is given by
\begin{align*}
	p_{\mathrm{Index}}(i\mid \theta,\rho, \ell,R)\ \propto\ \mu(\theta+\rho hi)\,\indc{R-2^\ell+1\leq i\leq R }\;.
\end{align*}
The indicator function ensures that $i$ is chosen within the valid index range. Altogether, the space containing the state $\theta$ and all auxiliary variables $\rho, \ell, R, i$ is \[
\mathcal Z=\R^d\times\R^d\times\mathbb{N}\times \mathbb{N}\times \mathbb{Z}
\] with elements denoted as $z=(\theta,\rho, \ell, R, i)$.  The canonical measure $\zeta$ on $\mathcal Z$ is the product of Lebesgue measure in the first two components and the counting measure on the latter three.
We define an enlarged target distribution on $\mathcal Z$ with a density, relative to $\zeta$, given by
\begin{align*}
	p_{\mathrm{joint}}(z)\ \propto\ \mu(\theta)\cdot\tau(\rho)\cdot p_{\mathrm{Orbit}} (\ell,R\mid\theta,\rho)\cdot  p_{\mathrm{Index}}(i\mid \theta,\rho, \ell,R)\;.
\end{align*}
In the enlarged space $\mathcal Z$, consider the map $\Psi:\mathcal Z\to\mathcal Z$ defined by
\[ \Psi(z)\ =\ (\theta+\rho h i,\,\rho,\,\ell,\,R-i,\,-i)\;. \]
The NURS transition $\theta\mapsto\theta'=\theta+\rho h i$ corresponds to the first component of $\Psi$.  Additionally, the marginal distribution of $p_{\mathrm{joint}}(z) \,\zeta(\rd z)$ in the first component equals the target distribution $\mu$.
Therefore, reversibility of NURS with respect to the target $\mu$ follows from reversibility of $\Psi$ with respect to $p_{\mathrm{joint}}(z) \,\zeta(\rd z)$.
This however can be checked via a direct computation using:
\begin{lemma}\label{lem:Psi}
    The map $\Psi$ is a $\zeta$-preserving involution, i.e., it holds that $\Psi\circ\Psi=\mathrm{id}$, such that $p_{\mathrm{joint}}\circ\Psi=p_{\mathrm{joint}}$.
\end{lemma}
Specifically, for any measurable sets $A,B\subset\mathcal Z$, we have
\begin{align*}
    \int_{\mathcal Z} \Indc{z\in A, \Psi(z)\in B}\,p_{\mathrm{joint}}(z)\,\zeta(\rd z)
    \ &=\ \int_{\mathcal Z} \Indc{\Psi\circ\Psi(z)\in A, \Psi(z)\in B }\,p_{\mathrm{joint}}(z) \,\zeta(\rd z) \\
    &=\ \int_{\mathcal Z} \Indc{\Psi(z)\in A, z \in B }\,p_{\mathrm{joint}}\circ\Psi(z)\, \zeta(\rd z)\\
    &=\ \int_{\mathcal Z} \Indc{z\in B, \Psi(z)\in A }\,p_{\mathrm{joint}}(z)\, \zeta(\rd z)\;,
\end{align*}
finishing the proof.
\end{proof}

\begin{proof}[Proof of Lemma~\ref{lem:Psi}]
The map $\Psi$ is an involution, since for all $z\in\mathcal Z$,
\begin{align*}
	\Psi\circ\Psi(z)\ =\ \Psi(\theta+\rho h i, \rho,\ell, R-i, -i )\ =\ (\theta, \rho, \ell, R, i) \ =\ z\;.
\end{align*}
Furthermore, except for the reflection of the last component, $\Psi$ acts as a translation on all other components.  As both Lebesgue and counting measure are invariant under translations and reflections, the product measure $\zeta$ is also invariant under $\Psi$.

We are left to show $p_{\mathrm{joint}}\circ\Psi=p_{\mathrm{joint}}$.
By definition, we have
    \begin{align*}
	p_{\mathrm{joint}}(z)\ \propto\ \mu(\theta)\cdot\tau(\rho)\cdot p_{\orbit}(\ell,R\mid\theta,\rho) \cdot\mu(\theta+\rho hi)\;.
    \end{align*}
    Similarly, the joint density at $\Psi(z)= (\theta+\rho h i,\,\rho,\,\ell, \,R-i,\,-i)$ equals
    \begin{align*}
        p_{\mathrm{joint}}\circ\Psi(z)\ \propto\ \mu(\theta+\rho h i)\cdot \tau(\rho)\cdot p_{\orbit}(\ell,R-i\mid \theta+\rho h i,\rho) \cdot \mu(\theta)\;.
    \end{align*}
    The claim therefore follows from the symmetry
    \begin{align}
    \label{equ: orbit symmetry}
        p_{\orbit}(\ell,R\mid\theta,\rho )\ =\ p_{\orbit}(\ell,R-i\mid \theta+\rho h i, \rho)\;,
    \end{align}
    which we establish in the remainder of this proof.
    
    Given a state $\theta \in \mathbb{R}^d$, direction $\rho \in \mathbb{R}^d$, and lattice spacing $h>0$, the orbit $\calO$ is built by a sequence of random forward and backward doubling steps. 
    Let $(B_1, \dots, B_M) \sim\mathrm{Bernoulli}(1/2)^{\otimes M}$ denote the random variables determining the direction of each doubling step, where $B_k=1$ corresponds to doubling in the forward direction and $B_k=0$ corresponds to doubling in the backward direction. After $k$ doublings, the resulting orbit is denoted by \[
    \calO_k=(\theta^{\leftmost}_k,\ldots,\theta^{\rightmost}_k)
    \]  where the indices of the endpoints are defined as
    \begin{equation*}
        \begin{alignedat}{2}
            \theta^{\leftmost}_k \ &= \ \theta - a_k \,h\, \rho  &\qquad \text{where } \quad &a_k \ = \ \sum_{i=1}^k (B_i-1)2^{i-1}\;,\quad\text{and}\\
            \theta^{\rightmost}_k \ &= \ \theta + b_k \,h\, \rho  &\qquad \text{where } \quad &b_k \ = \ \sum_{i=1}^k B_i 2^{i-1}\;.
        \end{alignedat}
    \end{equation*}
    In other words, the rightmost point of $\calO_k$ is obtained by taking $b_k$ steps from $\theta$ in the direction $\rho$, and the leftmost point of $\calO_k$ is obtained by taking $a_k$ steps in the opposite direction. Since the stopping condition is deterministic, given $\theta$ and $\rho$, the only randomness entering the orbit selection comes from the Bernoulli random variables.
    
    Suppose $|\calO|=2^\ell$, meaning the doubling procedure stops after $\ell$ doublings. If the rightmost point of $\calO$ is $R$ steps away from $\theta$, then the binary expansion $R=\sum_{i=1}^\ell B_i 2^{i-1}$ uniquely determines $(B_1,\ldots,B_\ell)$.
    Hence, there is exactly one sequence $(B_1,\ldots,B_\ell)\in\{0,1\}^\ell$ that generates the orbit $\calO$. 
     
     If $\ell=M$, then the maximum number of doublings is reached, and $\calO$ does not satisfy the sub-stopping condition, i.e., neither $\calO$ nor any of its sub-orbits (as defined in \eqref{eq:sub-orbits}) satisfies the stopping condition.  
     
     If $\ell<M$, two scenarios are possible:  \begin{itemize}
  \item  $\calO$ meets the stopping condition but no sub-orbit of $\calO$ does, or
  \item $\calO$ does not meet the sub-stopping condition but its extension $\calO_{\ell+1}^{\ext}$ does.
  \end{itemize} 
  
  Therefore, the probability of selecting $\calO$ equals
    \begin{equation}\label{equ: orbit probability}
    \begin{aligned}
    p_{\orbit}(\calO\mid\theta,\rho)\ =\ \frac{1}{2^\ell}\cdot \mathbf{1} & \bigr\{\ell=M \text{ or } \calO \text{ meets the stopping condition} \\
   & \quad \text{ or } \calO_{\ell+1}^{\ext} \text{ meets the sub-stopping condition}  \bigr\} \;.
    \end{aligned}
    \end{equation} 
  
    In particular, the stopping condition only depends on the orbit, and not on any particular starting state. Therefore, the probability of selecting $\calO$ given in \eqref{equ: orbit probability} satisfies
    \begin{align*}
        p_{\orbit}(\calO\mid\tilde\theta,\rho)\ =\ p_{\orbit}(\calO\mid \theta,\rho)\quad \text{for all $\tilde\theta\in\calO$.}
    \end{align*} 

    For $\tilde\theta=\theta+\rho h i\in\calO$, the rightmost endpoint of $\calO$ is $R-i$ steps away from $\tilde\theta$. That is, starting from $(\theta+\rho h i, \rho)$, the orbit $\calO$ corresponds to $(\ell, R-i)$.
    Therefore, the last display corresponds to \eqref{equ: orbit symmetry}.
\end{proof}

\section{Conclusion \& Outlook}
\label{sec:outlook}

This work introduces the No-Underrun Sampler (NURS), a locally adaptive, gradient-free MCMC method that combines coordinate-free direction selection, orbit-based exploration, a No-Underrun condition and categorical state selection to enable efficient sampling.  NURS provides a practical alternative to Hit-and-Run, a theoretically appealing but computationally impractical coordinate-free Gibbs sampler.

The theoretical analysis in the paper prove key properties of NURS, including reversibility, Wasserstein contraction, and connections to Hit-and-Run and Random Walk Metropolis, providing a framework for understanding its behavior and guiding its parameter tuning.  Empirical tests on Neal’s funnel show that while NURS moves diffusively in narrow regions, its ballistic movement in broader regions offsets this, enabling efficient sampling across the distribution’s different scales. 

Several promising directions remain for further exploration:
\begin{itemize}
\item Incorporating local lattice size adaptivity could remove the need for tuning the lattice size.  The GIST framework, recently developed for incorporating step-size adaptivity in the No-U-Turn Sampler \cite{BouRabeeCarpenterMarsden2024,BouRabeeCarpenterKleppeMarsden2024}, provides a natural framework for such extensions. 
\item Developing optimized implementations, including parallelization strategies, that could fully leverage NURS's inherent parallelizability for large-scale applications.
\item Proving convergence guarantees that encompass general target distributions would deepen the understanding of NURS. Promising approaches include conductance \cite{andrieuAAP,chen2022hit} and coupling \cite{BouRabeeOberdoerster2023}.
\item Locally adapting the direction distribution $\tau$ using an ensemble of chains, as in the ensemble samplers with affine invariance \cite{foreman2013emcee,goodman2010ensemble}.
\end{itemize}

\subsection*{Acknowledgements}

Thank you to Andreas Eberle and Sam Power for useful discussions and suggestions.  N.~Bou-Rabee was partially supported by NSF grant  No.~DMS-2111224.
S.~Oberd\"orster was supported by the Deutsche Forschungsgemeinschaft (DFG, German Research Foundation) under Germany’s Excellence Strategy – EXC-2047/1 – 390685813.

\bibliographystyle{imsart-number}
\bibliography{all.bib}  

\newpage

\appendix
\section{Pseudocode Implementation of NURS} \label{sec:pseudocode}

\begin{mylisting}[ht]
    \begin{flushleft}
    $\texttt{\bfseries NURS}\!\left(\theta,  \epsilon, h,  M \right)$
    \vspace*{2pt}
    \hrule
    \vspace*{2pt}
    Inputs:
    \begin{tabular}[t]{ll}
    $\mu:\mathbb{R}^d \rightarrow (0, \infty)$ & non-normalized target density
    \\
    $\theta \in \mathbb{R}^d$ & starting point
    \\
    $\epsilon \ge 0$ & density threshold
    \\
    $h > 0$ & lattice spacing
    \\
    $M \in \mathbb{N}$ & maximum number of doublings (tree depth)
    \\
    \end{tabular} 
    \\[2pt]
    Return: 
    \begin{tabular}[t]{ll}
    $\tilde{\theta} \in \mathbb{R}^d$ & next draw
    \end{tabular}
    \vspace*{4pt}
    \hrule
    \vspace*{6pt}

    $\rho \sim \Unif(\mathbb{S}^{d-1})$

    $u\sim \Unif((0,1))$

    $s \sim \Unif([-h/2,h/2])$

    if $u \leq \dfrac{\mu(\theta+s \rho)}
                            {\mu(\theta)}:$
                            
    \qquad $\tilde \theta=\theta + s \rho$
    
    else: 
    
    \qquad $\tilde \theta=\theta$
    
    $\calO=( \tilde \theta )$ 
    
    $B \sim \Unif(\{0,1\}^M)$

    for $k$ from $1$ to $M$:

    \qquad $\calO^{\old} = \calO$
    
    \qquad if $B_k=1$:

    \qquad \qquad $\calO^{\ext} = \texttt{extend-orbit}(\theta^{\rightmost}, \rho, h, 2^{k-1}) $
    \\
    \qquad \qquad $\calO =\calO^{\old} \odot \calO^{\ext} $

    \qquad else:\hfill

   \qquad \qquad $\calO^{\ext} = \texttt{extend-orbit}(\theta^{\leftmost}, \rho,-h,2^{k-1}) $
   \\
   \qquad \qquad $\calO =\calO^{\ext} \odot  \calO^{\old}  $

    \qquad if $\texttt{sub-stop}(\calO^{\ext},\epsilon,h)$:
    
    \qquad \qquad break

    \qquad  $u\sim  \Unif((0,1))$

    \qquad if $u \leq \dfrac{\sum_{\breve{\theta} \in \calO^{\ext}} \mu(\breve{\theta})}
                            {\sum_{\breve{\theta} \in \calO} \mu(\breve{\theta})}:$
    

    \qquad\qquad $\tilde\theta \sim \cat( \calO^{\ext}, \mu)$

    \qquad if $\texttt{stop}(\calO,\epsilon,h)$:
    
    \qquad \qquad break

    return $\tilde \theta$\hfill
    \vspace*{6pt}
    \hrule
    \caption{\it The No-Underrun Sampler.  
    The \emph{\texttt{extend-orbit()}} function is given in Listing~\ref{lst:extend-orbit}, \emph{\texttt{stop()}} in Listing~\ref{lst:stop}, and \emph{\texttt{sub-stop()}} in Lising~\ref{lst:sub-stop}.}
    \label{lst:nurs}
    \end{flushleft}
\end{mylisting}

\begin{mylisting}[ht]
    \begin{flushleft}
    $\texttt{\bfseries extend-orbit}\!\left(\theta, \rho, h, n\right)$
    \vspace*{2pt}
    \hrule
    \vspace*{2pt}
    Inputs:
    \begin{tabular}[t]{ll}
      $\theta \in \mathbb{R}^d$ & initial position
      \\
      $\rho \in \mathbb{R}^d$ & direction
      \\
      $h \ne 0$ & lattice spacing
      \\
      $n \in \mathbb{N}$ & number of steps
    \end{tabular} 
    \\[2pt]
    Return: 
    \begin{tabular}[t]{ll}
      $\left( \mathbb{R}^d \right)^n$ & orbit
    \end{tabular}
    \vspace*{4pt}
    \hrule
    \vspace*{8pt}
    if $h > 0$

    \qquad return $\theta + \rho \cdot h \cdot (1, 2, \ldots, n)$

    else

    \qquad return $\theta + \rho \cdot h \cdot (n, n-1, \ldots, 1)$

    \vspace*{4pt}
    \hrule
    \vspace*{4pt}
    \caption{\it Extend the orbit from a specified starting point for the specified number of steps with the specified direction and lattice spacing.  Here, a list is treated like a vector for arithmetic purposes. \hfill \null}
    \label{lst:extend-orbit}
    \end{flushleft}
\end{mylisting}

\begin{mylisting}[ht]
    \begin{flushleft}
    $\texttt{\bfseries stop}\!\left(\calO,\epsilon,h\right)$
    \vspace*{2pt}
    \hrule
    \vspace*{2pt}
    Inputs:
    \begin{tabular}[t]{ll}
     $\calO= \left(\theta^{\leftmost}, \dots, \theta^{\rightmost}\right)$  
     & orbit, with $|\calO|$ a power of 2
     \\
     $\epsilon \ge 0$ & density threshold \\
    $h > 0$ & lattice spacing
    \end{tabular} 
    \\[2pt]
    Return: 
    \begin{tabular}[t]{ll}
    \texttt{boolean} & \texttt{True} if $\calO$ satisfies the stopping criterion
    \end{tabular}
    \vspace*{4pt}
    \hrule
    \vspace*{8pt}
return $\max\bigr(\mu(\theta^{\leftmost}),\,\mu(\theta^{\rightmost}) \bigr) \ \leq \ \epsilon \;h \, \sum_{\breve{\theta} \in \calO} \mu(\breve{\theta})$ 
    \vspace*{4pt}
    \hrule
    \vspace*{4pt}
    \caption{\it Check if the orbit $\calO$ satisfies the No-Underrun stopping condition in \eqref{eq:no-Underrun}. \hfill\null}
    \label{lst:stop}
    \end{flushleft}
\end{mylisting}
    
\begin{mylisting}[ht]
    \begin{flushleft}
    $\texttt{\bfseries sub-stop}(\calO,\epsilon,h)$
    \vspace*{2pt}
    \hrule
    \vspace*{2pt}
    Inputs:
    \begin{tabular}[t]{ll}
    $\calO = \calO^{\leftmost} \odot \calO^{\rightmost}$    
    & orbit, with $|\calO^\leftmost| = |\calO^\rightmost|$ a power of 2
    \\
     $\epsilon \ge 0$ & density threshold \\
    $h > 0$ & lattice spacing
    \end{tabular} 
    \\[2pt]
    Return:
    \begin{tabular}[t]{ll}
    \texttt{boolean} & \texttt{True} if $\calO$ or a subtree satisfies the stopping criterion
    \end{tabular}
    \vspace*{4pt}
    \hrule
    \vspace*{8pt}
    if $\text{length}(\calO) < 2$: \ return \texttt{False}
    \\[4pt]
 return $ \texttt{stop}(\calO, \epsilon,h)
        ~~\texttt{or}~~ \texttt{sub-stop}(\calO^{\leftmost}, \epsilon,h)
        ~~ \texttt{or}~~ \texttt{sub-stop}(\calO^{\rightmost}, \epsilon,h)$
    \vspace*{8pt}
    \hrule
    \vspace*{4pt}
    \caption{\it Check if the orbit or any of its sub-orbits obtained by repeated halving, as defined in \eqref{eq:sub-orbits}, satisfy the stopping condition.
    Function \emph{\texttt{stop()}} is defined in Listing~\ref{lst:stop}.\hfill\null}
    \label{lst:sub-stop}
    \end{flushleft}
\end{mylisting}

\end{document}